\documentclass[11pt]{amsart}
\usepackage{amscd,amsmath,amsthm,amssymb}
\usepackage{fullpage}
\usepackage{amscd,amsmath,amsthm,amssymb}
\usepackage{comment}
\usepackage[all]{xy}
\usepackage{datetime}
\usepackage{color}
\usepackage[english]{babel}
\usepackage[utf8x]{inputenc}
\usepackage[T1]{fontenc}
\usepackage{hyperref}
\usepackage{rotating}

\usepackage[foot]{amsaddr} 
\usepackage{booktabs}

\usepackage[a4paper,top=3cm,bottom=2cm,left=3cm,right=3cm,marginparwidth=1.75cm]{geometry}

\usepackage{amsmath}
\usepackage{amssymb}
\usepackage{amsthm}
\usepackage{verbatim}

\usepackage{tikz}
\usetikzlibrary{calc}

\usepackage{ifpdf}
\ifpdf
\usepackage{pdflscape}
\else
\usepackage{lscape}
\fi

\def\NZQ{\mathbb}               

\def\KK{{\NZQ K}}

\def\frk{\frak}               

\def\Phi{{\frk n}}
\def\Phi{{\frk N}}
\def\MC{{\mathcal C}}
\def\MH{{\mathcal H}}
\def\MI{{\mathcal I}}

\def\ML{{\mathcal L}}

\def\MX{{\mathcal X}}
\def\MY{{\mathcal Y}}
\def\MZ{{\mathcal Z}}

\newcommand{\R}{\mathbb{R}}
\newcommand{\C}{\mathbb{C}}

\newcommand{\I}{\mathcal{I}}
\newcommand{\LL}{\mathcal{L}}
\newcommand{\ID}{J_{\MX,\Delta}}

\newcommand{\la}{\langle}
\newcommand{\ra}{\rangle}

\newcommand{\lv}{\lvert}
\newcommand{\rv}{\rvert}

\newcommand{\plex}{\prec_{\rm lex}}

\newcommand{\bolda}{{\bf a}}
\newcommand{\boldb}{{\bf b}}
\newcommand{\init}{{\rm in}_{\prec}}

\newcommand{\matrank}{{\rm rank}}

\def\opn#1#2{\def#1{\operatorname{#2}}} 
\opn\chara{char} \opn\length{\ell} \opn\pd{pd} \opn\rk{rk}
\opn\projdim{proj\,dim} \opn\injdim{inj\,dim} \opn\rank{rank}
\opn\depth{depth} \opn\grade{grade} \opn\height{height}
\opn\embdim{emb\,dim} \opn\codim{codim}

\opn\Tr{Tr} \opn\bigrank{big\,rank}
\opn\superheight{superheight}\opn\lcm{lcm}
\opn\trdeg{tr\,deg}
\opn\reg{reg} \opn\lreg{lreg} \opn\ini{in} \opn\lpd{lpd}
\opn\size{size} \opn\sdepth{sdepth}
\opn\link{link}\opn\fdepth{fdepth}\opn\lex{lex}
\opn\div{div} \opn\Div{Div} \opn\cl{cl} \opn\Cl{Cl}
\opn\Spec{Spec} \opn\Supp{Supp} \opn\supp{supp} \opn\Sing{Sing}
\opn\Ass{Ass} \opn\Min{Min}\opn\Mon{Mon}
\opn\Ann{Ann} \opn\Rad{Rad} \opn\Soc{Soc}
\opn\Im{Im} \opn\Ker{Ker} \opn\Coker{Coker} \opn\Am{Am}
\opn\Hom{Hom} \opn\Tor{Tor} \opn\Ext{Ext} \opn\End{End}
\opn\Aut{Aut} \opn\id{id}

\opn\nat{nat}
\opn\pff{pf}
\opn\Pf{Pf} \opn\GL{GL} \opn\SL{SL} \opn\mod{mod} \opn\ord{ord}
\opn\Gin{Gin} \opn\Hilb{Hilb}\opn\sort{sort}
\opn\aff{aff} \opn\con{conv} \opn\relint{relint} \opn\st{st}
\opn\lk{lk} \opn\cn{cn} \opn\core{core} \opn\vol{vol}
\opn\link{link} \opn\star{star}\opn\lex{lex}\opn\set{set}
\opn\gr{gr}

\def\pot#1#2{#1[\kern-0.28ex[#2]\kern-0.28ex]}

\opn\dirlim{\underrightarrow{\lim}}
\opn\inivlim{\underleftarrow{\lim}}
\let\iso=\cong

\def\Implies{\ifmmode\Longrightarrow \else
        \unskip${}\Longrightarrow{}$\ignorespaces\fi}
\def\implies{\ifmmode\Rightarrow \else
        \unskip${}\Rightarrow{}$\ignorespaces\fi}
\def\iff{\ifmmode\Longleftrightarrow \else
        \unskip${}\Longleftrightarrow{}$\ignorespaces\fi}

\let\:=\colon
\newtheorem{Theorem}{Theorem}[section]
\newtheorem{Lemma}[Theorem]{Lemma}
\newtheorem{Corollary}[Theorem]{Corollary}
\newtheorem{Proposition}[Theorem]{Proposition}
\newtheorem{Remark}[Theorem]{Remark}
\theoremstyle{remark}
\newtheorem{Example}[Theorem]{Example}

\newtheorem{Definition}[Theorem]{Definition}

\let\epsilon\varepsilon
\let\kappa=\varkappa
\opn\dis{dis}
\def\pnt{{\raise0.5mm\hbox{\large\bf.}}}

\opn\Lex{Lex}

\newcommand\Perp{\protect\mathpalette{\protect\independenT}{\perp}}
\def\independenT#1#2{\mathrel{\rlap{$#1#2$}\mkern2mu{#1#2}}}
\newcommand{\ind}[3][]{\left.#2 \Perp\!_{#1}\, #3 \inD}
\newcommand{\inD}[1][\relax]{\def\argone{#1}\def\temprelax{\relax}
  \ifx\argone\temprelax\right.\else\,\middle|#1\right.{}\fi}

\DeclareMathOperator*{\bigtimes}{\textnormal{\Large $\times$}}

\begin{document}

\title{Conditional independence ideals \\ with hidden variables}

\author{Oliver Clarke$^{1}$}
\address{$^{1}$University of Bristol, School of Mathematics, BS8 1TW, Bristol, UK}
\email{oliver.clarke@bristol.ac.uk}
\author{Fatemeh Mohammadi$^{2,3}$}
\noindent\address{$^{2}$Department of Mathematics: Algebra and Geometry, Ghent University, 9000 Gent, Belgium}
\address{$^{3}$The Department of Mathematics and Statistics, Faculty of Science and Technology, UiT – The Arctic University of Norway, 9037 Troms, Norway}
\email{fatemeh.mohammadi@ugent.be}
\author{Johannes Rauh$^{4}$}
\address{$^{4}$Max Planck Institute for Mathematics in the Sciences, Leipzig, Germany}
\email{jrauh@mis.mpg.de}
\begin{abstract}
We study a class of determinantal ideals that are related to conditional independence (CI) statements with hidden
variables.  Such CI statements correspond to determinantal conditions on
a matrix whose entries are probabilities of events involving the observed random variables.
We focus on an example that generalizes the CI ideals of the intersection axiom.
In this example, the minimal primes are again determinantal
ideals, which is not true in general.
\end{abstract}

\maketitle

\section{Introduction}

This work is concerned with the study of conditional independence (CI) statements with hidden variables through the lens of algebraic statistics and commutative algebra.
We are led to a class of ideals generated by minors of a matrix of indeterminates, where the indeterminates correspond to probabilities.  The classical case of CI ideals without hidden variables corresponds to a class of ideals generated by 2-minors.

\subsection{Structure of the paper.}

In Section~\ref{sec:CI_ideals}, we explain the setting of conditional independence ideals that
motivates our study.  In Section~\ref{sec:CI_with_hidden}, we show how to define conditional
independence ideals with hidden variables.  The determinantal ideals that we are interested in are
defined in Section~\ref{sec:our_CI_ideals}.  In Section~\ref{sec:determ-hyper-ideals}, we show that
the ideals can also be defined as determinantal hyperedge ideals.  The reader who is not interested
in the statistical motivation may skip Sections~\ref{sec:CI_ideals}
to~\ref{sec:our_CI_ideals}.

The main results will be presented in Section~\ref{sec:main-results}, with proofs given in
Section~\ref{sec:proofs}.  The final Section~\ref{sec:gen} discusses possible
generalizations.

\subsection{Conditional independence ideals}
\label{sec:CI_ideals}

Conditional independence is an important tool in statistical modelling~\cite{Studeny05:Probabilistic_CI_structures}.
For example, it gives an interpretation to graphical models, both undirected (i.e.\ Markov fields) and directed (i.e.\ Bayesian networks) \cite{MDLW18:Handbook_of_graphical_models, caines2022lattice}.

Suppose a family of random variables satisfies a list of conditional independence statements.
Given a sub-family of these random variables, it is in general difficult to say which constraints are implied by these CI statements on this sub-family.
This situation arises when some of the random variables are either hidden (i.e.\ not observed) or not of interest. Here we are interested in the constraints that are satisfied by the observed variables (or the variables of interest) alone.
For example, in causal reasoning, it is important to know what constraints on the observed variables are caused by hidden variables~\cite{Steudel-Ay}.

For graphical models, the problem of describing the induced model on the observed variables is well-studied.
The joint distribution on the observed variables is constrained by both equalities and inequalities.
When all random variables are finite (as we will assume throughout this paper),
conditional independence can be characterized by polynomial constraints on the joint probabilities of all variables.  Therefore, the equality and inequality constraints on the joint probabilities of the observed variables can also be formulated algebraically; that is, the set of joint probability distributions forms a semi-algebraic set~\cite{DrtonSturmfelsSullivant09:Algebraic_Statistics, Sullivant}.
The models that have been studied so far include hidden Markov models, mixture models or restricted Boltzmann machines~\cite{AllmanRhodesTaylor:Semialgebraic_phylogenetic_Markov_model,SeigalMontufar17:RBM_32}.
Some, but not all, of the constraints among the observed variables have an interpretation in terms of conditional independence.
For example, the Verma constraints can be interpreted as coming from a conditional independence statement \cite{RichardsonRobinsShpitser11:Nested_Markov_Properties}.

In algebraic statistics, the ideals corresponding to conditional independence statements without hidden variables are called \emph{conditional independence ideals}; they have been studied in \cite{Fink,herzog2010binomial,Rauh,SwansonTaylor11:Minimial_Primes_of_CI_Ideals}.
Here, we study a class of ideals that are derived from conditional independence statements with hidden variables.  While ordinary conditional independence ideals are generated by 2-minors of a matrix of probabilities, CI statements with hidden variables introduce constraints that correspond to higher minors.

\subsection{CI statements with hidden variables}

\label{sec:CI_with_hidden}

Consider two random variables $X,Y$ with finite state spaces~$\MX,\MY$.  The joint distribution of $X,Y$ can be
identified with a $\MX\times\MY$-matrix $P=(p_{x,y})_{x\in\MX,y\in\MY}$, where $p_{x,y}=P(X=x,Y=y)$.  The variables $X$
and $Y$ are \emph{independent} if and only if $P$ has rank one~\cite{DrtonSturmfelsSullivant09:Algebraic_Statistics, Sullivant}.

In the presence of a third random variable $Z$ with state space~$\MZ$, the joint probability distribution becomes a
tensor $P=(p_{x,y,z})_{x\in\MX,y\in\MY,z\in\MZ}$.  The variables $X$ and $Y$ are \emph{independent given $Z$} if and only if
for each $z\in\MZ$ the matrix $P_{z} := (p_{x,y,z})_{x\in\MX,y\in\MY}$ has rank one.  In this case we write
$\ind XY[Z]$ and we call this expression a \emph{conditional independence (CI) statement.}

Suppose that $Z$ is a hidden (unobserved) variable, and that we only have access to the marginal
distribution of $X$ and~$Y$, which is equal to the marginal tensor~$P^{X,Y}=\sum_{z\in\MZ}P_{z}$.  If $X$ and $Y$ are
independent given~$Z$, then $P^{X,Y}$ has rank at most~$|\MZ|$.
More precisely, $P^{X,Y}$ has \emph{non-negative rank} at most $|\MZ|$; that is, $P^{X,Y}$ can be written as a non-negative linear combination of
non-negative rank-one matrices.
Conversely, any (non-negative real) matrix $P^{X,Y}$ of non-negative rank at most $|\MZ|$ that satisfies the normalization condition~$\sum_{x,y}P^{X,Y}_{x,y}=1$ arises in this way, as the $(X,Y)$-marginal of a joint distribution tensor $P$ of three variables $X,Y,Z$ that satisfies $\ind XY[Z]$.

The set of matrices of non-negative rank at most $r$ form a semi-algebraic set.  A characterization
of this semi-algebraic set for $r=2$ is given in \cite{allman2015tensors}.  For $r>2$, the
semi-algebraic conditions are not known in general.  It is known that its Zariski closure equals the
set of all matrices of rank at most $r$, and it is described by the determinantal ideal of all $(r+1)$-minors
of~$P$.

When combining several CI statements with hidden variables, the set of corresponding joint
distributions of the observed variables will again be a semi-algebraic set. In this paper, we
restrict attention to a subset of the equality constraints consisting of rank constraints.  While
these rank constraints are in general not even enough to describe the Zariski closure, they do
provide valuable insights.



\begin{Example}
  \label{ex:phylogenetic}
  The invariants of edges that appear in the implicit description of phylogenetic models encode rank
  conditions that arise from CI statements with hidden variables~\cite{AllmanRhodes07:Reconstruction_Evolution}.  Suppose that
  $Y_{1},\dots,Y_{n_{o}}$ are observed base pairs of extant species at a specific location (assuming
  that the alignment is known). Each edge $e$ in a phylogenetic tree induces a split of the observed
  variables, say $Y_{1},\dots,Y_{k}|Y_{k+1},\dots,Y_{n_{o}}$.  This split corresponds to a CI
  statement $\ind{\{Y_{1},\dots,Y_{k}\}}{\{Y_{k+1},\dots,Y_{n_{o}}\}}[H_{e}]$, where $H_{e}$ is one
  of the unobserved nodes of~$e$.  For a single base pair, $H_{e}$ can take four values, which correspond to the four types of bases found in a DNA molecule (adenine, cytosine, guanine, and thymine), and so
  the invariants of edges consist of 5-minors of a matrix of probabilities.

  The invariants of edges do not give a full description of the phylogenetic model, nor its Zariski closure.
  Nevertheless, as shown in~\cite{CasanellasFernandez11:Relevant_phylogenetic_invariants}, the invariants of edges
  contain enough information to distinguish between different tree topologies.
\end{Example}

\subsection{Ideals of CI statements}

\label{sec:our_CI_ideals}

The ideals that we are interested in arise from the following situation:
Consider three observed variables $X,Y_{1},Y_{2}$, taking values in the finite sets $\MX,\MY_{1},\MY_{2}$ of cardinalities $|\MX|=d$, $|\MY_{1}|=k$, $|\MY_{2}|=\ell$, and two hidden variables $H_{1},H_{2}$, taking values in the finite sets $\MH_{1},\MH_{2}$ of cardinalities $|\MH_{1}|=s-1$,  $|\MH_{1}|=t-1$.

The joint distribution of the observed variables can be identified with a non-negative matrix
$P\in\R^{\MX\times\MY}$, where $\MY=\MY_{1}\times\MY_{2}$.  The matrix $P$ has $d$ rows and
$k\ell$ columns, and its entries $p_{x,(y_{1},y_{2})}$ sum to one.  The row indices
correspond to states $x\in\MX$, and the column indices correspond to joint states
$(y_1,y_{2})\in\MY$.
For each $i \in \MY_1$, we define $R_i = \{(i, j) \in \MY : j \in \MY_2 \}$ and for each $j \in \MY_2$ we define $C_j = \{(i,j) \in \MY : i \in \MY_1 \}$.

Let $K$ be a field and $R$ be the polynomial ring over $K$ in the variables $p_{x,(y_{1},y_{2})}$ corresponding to the entries of $P$.
From the statistical point of view, it makes sense to assume that $K=\R$ or $K=\C$ (for algebraic convenience), but as it turns out, our algebraic results are independent of the choice of the field.
For each $F \subseteq\MY$ denote by $P_{F}$ the submatrix of $P$ consisting of the columns indexed
by~$F$.
Then we define two ideals in~$R$:
\begin{align*}
  J_{\ind{X}{Y_{1}}[\{Y_{2},H_{1}\}]} &= \big\langle
    \text{$s$-minors of } P_{C_{j}} \text{ for all }j\in\MY_2
  \big\rangle, \\
  J_{\ind{X}{Y_{2}}[\{Y_{1},H_{2}\}]} &= \big\langle
    \text{$t$-minors of } P_{R_{i}} \text{ for all }i\in\MY_1
  \big\rangle.
\end{align*}
Finally, for $\MC = \big\{\ind{X}{Y_{1}}[\{Y_{2},H_{1}\}], \ind{X}{Y_{2}}[\{Y_{1},H_{2}\}]\big\}$, we define
\begin{equation*}
  J_{\MC} = J_{\ind{X}{Y_{1}}[\{Y_{2},H_{1}\}]} + J_{\ind{X}{Y_{2}}[\{Y_{1},H_{2}\}]}.
\end{equation*}
Note that the ideals $J_{\ind{X}{Y_{1}}[\{Y_{2},H_{1}\}]}$, $J_{\ind{X}{Y_{2}}[\{Y_{1},H_{2}\}]}$ and $J_{\MC}$ not only depend on the CI statements, but also on $d$, $k$,
  $\ell$, $s$ and~$t$.  However, these additional parameters will usually be fixed and clear from
  the context, so we omit them from the notation.

By what has been said in Section~\ref{sec:CI_with_hidden}, the following holds: \emph{if $X,Y_{1},Y_{2},H_{1},H_{2}$ are random variables that satisfy the CI statements in $\MC$, then the joint distribution $P$ of the observed variables $X,Y_{1},Y_{2}$ lies in the vanishing set of~$J_{\MC}$.}

\begin{Example}
  Let 
  $\MX = \{1,2,3\} = \MY_{1} = \MY_{2}$.  The matrix $P\in\R^{\MX\times\MY}\cong\R^{3\times 9}$
  has the form
  \begin{equation*}
    \left(
    \begin{tikzpicture}[baseline={([yshift=-\the\dimexpr\fontdimen22\textfont2\relax]
                    current bounding box.center)}] %
      \foreach \x in {1,2,3} {
        \foreach \y in {1,2,3} {
          \foreach \z in {1,2,3} {
            \node (P\x\y\z) at (4 * \z + 1.3 * \y, 4 - \x - 0.1 * \y) {$p_{\x,(\y,\z)}$};G
          }
        }
      }
    \end{tikzpicture}
    \right).
  \end{equation*}
  The small vertical displacement of the entries indicates that the
  matrix $P$ can be seen as a flattening of a 3-tensor, and it allows
  to conveniently mark the submatrices $P_{C_{j}}$ and $P_{R_{i}}$.

  The coloured lines in the following display indicate the two submatrices $P_{R_{1}}$ and~$P_{C_{3}}$:
  \begin{equation*}
    \left(
    \begin{tikzpicture}[baseline={([yshift=-\the\dimexpr\fontdimen22\textfont2\relax]
                    current bounding box.center)}] %
      \foreach \x in {1,2,3} {
        \foreach \y in {1,2,3} {
          \foreach \z in {1,2,3} {
            \node[inner sep=0pt] (P\x\y\z) at (4 * \z + 1.3 * \y, 4 - \x - 0.1 * \y) {$p_{\x,(\y,\z)}$};G
          }
        }
      }
      \begin{scope}[red,dashed]
        \begin{scope}[transform canvas={yshift=1.5mm}]
          \foreach \x in {1,2,3} {
            \draw (P\x11) -- (P\x12) -- (P\x13);
          }
        \end{scope}
        \foreach \x in {1,2} {
          \draw (P11\x) -- (P21\x) -- (P31\x);
        }
        \draw[transform canvas={xshift=-0.5mm}] (P113) -- (P213) -- (P313);
      \end{scope}
      \begin{scope}[green,dotted,line width=1.5]
        \foreach \x in {1,2,3} {
          \draw (P\x13) -- (P\x23) -- (P\x33);
          \draw (P1\x3) -- (P2\x3) -- (P3\x3);
        }        
      \end{scope}
    \end{tikzpicture}
    \right).
  \end{equation*}

  Suppose that $|\mathcal{H}_1|=1$.  In this case, the hidden variable $H_{1}$ is constant and can be omitted from all CI statements; i.e.\ $\ind{X}{Y_{1}}[\{Y_{2},H_{1}\}]$ is equivalent to $\ind{X}{Y_{1}}[Y_{2}]$
  \footnote{
  While this degenerate case is the case covered by our main results, we chose to keep $H_{1}$ in order to indicate in which direction we would like to generalize our results.}.
  The ideal $J_{\ind{X}{Y_{1}}[\{Y_{2},H_{1}\}]}$ is generated by all 2-minors of the matrix $P_{C_{3}}$ connected by dotted green lines 
  and the 2-minors of $P_{C_{1}}$ and $P_{C_{2}}$.  

  Similarly, suppose that $|\mathcal{H}_2|=2$.  Then $J_{\ind{X}{Y_{2}}[\{Y_{1},H_{2}\}]}$ is generated by the determinant of the matrix $P_{R_{1}}$ \textcolor{blue}connected by dashed red lines 
  and the determinants of $P_{R_{2}}$ and $P_{R_{3}}$. 
\end{Example}

Our main results below give properties of the ideal $J_{\MC}$ in the case
$2\leq k\leq \ell \leq d$, $s = 2$, $t = \ell$.  In particular, Theorem~\ref{thm:intersection} describes
its minimal primes.
  Before stating our main results, we show how $J_{\MC}$ can be seen as a determinantal hyperedge ideal.

\subsection{Determinantal hyperedge ideals}
\label{sec:determ-hyper-ideals}

Throughout we let $\MX$, $\MY$ be finite sets, and consider a matrix $P=(p_{x,y})_{x\in\MX,y\in\MY}$ of indeterminates.  Let $K$ be a field, we work over the polynomial ring $R = K[P]$.
Denote by $d=|\MX|$ the number of rows of~$P$.
For any
$A\subseteq\MX$, $B\subseteq\MY$ with $|A|=|B|$ denote by $[A|B]$ the determinant of the submatrix of $P$ consisting of
those entries with row indices in $A$ and with column indices in~$B$.  When $A=\MX$, then we simply write $[B]$ for $[A|B]$.
We also abbreviate $[a_{1},\dots,a_{k}|b_{1},\dots,b_{k}]$ for $\big[\{a_{1},\dots,a_{k}\}\big|\{b_{1},\dots,b_{k}\}\big]$.

Formally, the sign of this determinant depends on the ordering of the rows and columns.  Later we are only interested
whether this determinant vanishes or not, and so the ordering of the rows and columns is not important.  However, when
studying the Gr\"obner bases of our ideals, we consider ordered sets, and we assume that $\MX = [d]$.

\begin{Definition}
  A \emph{hypergraph} $\Delta$ with vertex set $\MY$ is a subset of the power set~$2^{\MY}$.
  The elements of a hypergraph are \emph{hyperedges}.
  The \emph{determinantal hyperedge ideal} of~$\Delta$ is
  \begin{equation*}
    \ID = \big\langle [A|B] : A\subseteq\MX, B\in\Delta, |A| = |B|  \big\rangle \subset R.
  \end{equation*}
\end{Definition}
Since hyperedges $F\in\Delta$ with $|F|>|\MX|=d$ do not contribute to $\ID$, we may assume that $|F|\le d$ for
all~$F\in\Delta$ when studying $J_{\MX,\Delta}$.

\begin{Example}
Let $\MX = [3]$ and $\Delta$ be the hypergraph on $\MY=[6]$,
\[
\Delta = \{123, 124, 356, 456, 34\}. 
\]
Then the ideal $\ID \subset K[p_{1,1}, \dots, p_{3,6}]$ is generated by the following minors of $P$,
\[
[123], [124], [356], [456], [12|34], [13|34], [23|34].
\]
\end{Example}

Given natural numbers $k, \ell, s, t$ we define $\MY_1 = [k], \MY_2 = [\ell]$ and
$\MY=\MY_{1}\times\MY_{2}$.  It is sometimes convenient to identify $\MY$ with $[k\ell]$ and to
arrange its entries in a matrix:
\begin{equation*}
  \MY \cong [k\ell] = (\MY_{i,j})_{i,j} = 
  \begin{pmatrix}
    1 & k+1 & \cdots & (\ell-1)k+1 \\
    2&  k+2 & \cdots &  (\ell-1)k+2\\
    \vdots & \vdots &  & \vdots  \\
    k &  2k & \cdots &  \ell k\\
  \end{pmatrix}.
\end{equation*}
The rows of (this matrix which represents) $\MY$ are $R_i = \{\MY_{i,j} \in \MY : j \in \MY_2 \}$ for each $i \in \MY_1$. The columns of $\MY$ are $C_j = \{\MY_{i,j} \in \MY : i \in \MY_1 \}$ for each $j \in \MY_1$.

We define the hypergraphs 
\[
\Delta^t = \{B \subseteq \MY : |B| = t \} \textrm{ and } 
\Lambda^s = \bigcup_{j \in \MY_2} \{B \subseteq C_j : |B| = s \} 
\]
which will be used to describe the minimal prime components of the ideal $J_{\MC}$.
The ideal $J_{\MC}$ is a determinantal hyperedge ideal. To be precise, $J_{\MC}=J_{\MX,\Delta^{s,t}}$, where 
\[
\Delta^{s,t} = \Lambda^s \cup \bigcup_{i \in \MY_1} \{B \subseteq R_i : |B| = t \}.
\]

\section{Main results}
\label{sec:main-results}

Let $\Delta=\Delta^{s,t}$. We begin by defining the ideals which will appear in the prime decomposition of $J_{\MC}$.
\begin{Definition}
  \begin{enumerate}
  \item Let $S$ be a subset of $\MY=[k\ell]$.
    We define the ideal
    $$I_S = \la p_{i,j} : i \in \{1, \dots, \ell\}, j \in S \ra + J_{\MX,\Lambda^s}.$$
  \item Define the ideal $I_{0} = J_{\MX,\Lambda^s \cup \Delta^{t}}$.
  \end{enumerate}
\end{Definition}

\begin{Example}\label{example:k2l3s2td3}
Let $k = 2, \ell = 3, s = 2$ and $t = d = 3$.
Then $\MY = 
\begin{pmatrix}
1   &3  &5  \\
2   &4  &6
\end{pmatrix}.$ We have,
\[
\Delta = \{12, 34, 56, 135,246\},
\quad\text{and}\quad P= 
\begin{pmatrix}
p_{1,1}     &p_{1,2}    &p_{1,3}   &p_{1,4}    &p_{1,5}    &p_{1,6}      \\
p_{2,1}     &p_{2,2}    &p_{2,3}   &p_{2,4}    &p_{2,5}    &p_{2,6}     \\
p_{3,1}     &p_{3,2}    &p_{3,3}   &p_{3,4}    &p_{3,5}    &p_{3,6}       \\
\end{pmatrix}.  
\]
The corresponding hyperedge ideal is
\[J_{[3],\Delta} =\left\la [12|12],[13|12],[23|12],[12|34],[13|34],[23|34],[12|56],[13|56],[23|56],[135],[246] \right\ra.
\]
We also have the subgraph $\Lambda^s = \Lambda^2 = \{12, 34, 56\}$ of $\Delta$ and so, 
\[
J_{[3],\Lambda^2} = \left\la [12|12],[13|12],[23|12],[12|34],[13|34],[23|34],[12|56],[13|56],[23|56] \right\ra \subset J_{[3], \Delta}.
\]
Consider $S = \{1, 4 \}$.  Then
$I_S = I_{14} = 
\left\la p_{1,1}, p_{2,1}, p_{3,1}, p_{1,4}, p_{2,4}, p_{3,4}
\right\ra + J_{[3], \Lambda^2}$,
 where 
$J_{[3], \Lambda^2} = \la [12|56], [13|56], [23|56] \ra$.
The edges of the hypergraph $\Delta^t$ are all $3$-subsets of $[6]$. So the ideal $I_0$ is generated by all $3$-minors of $P$ along with all $2$-minors in $J_{[3],\Lambda^2}$.
\end{Example}

We assume that $2\leq k\leq \ell \leq d, s = 2, t = \ell$. We are mostly interested in the ideals $I_{S}$ when $S$ belongs to
\begin{equation*}
  \LL = \big\{\{\MY_{i_1, j_1}, \dots, \MY_{i_k, j_k}\} : \lv \{i_1,\dots i_k\} \rv = k, \lv \{j_1,\dots,j_k \} \rv \ge 2 \big\}.
\end{equation*}
In this case, $J_{\MX,\Delta^{2,t}}\subseteq I_{S}$.  Moreover, $J_{\MX,\Delta^{2,t}}\subseteq I_{0}$, since
$\Delta^{2,t} \subseteq \Lambda^2\cup\Delta^{t}$.

It is easy to write down a minimal generating set of the ideal $I_0$. 
\begin{Lemma}\label{lemma:Generating}
Let $\MX = [d]$. Then $I_0$ is minimally generated by the following set of minors:
\begin{multline*}
G(I_0) = 
\big\{[A \mid B]: A\subset[d], B\in\Delta, |A|=|B|=2 \big\}
\\
\cup \big\{[A \mid B] : A\subset[d], B\subset[k\ell], |A|=|B|=t, \lv B \cap C_j \rv \le 1\ \text{for each } j  \big\}.
\end{multline*}
\end{Lemma}

Let $\plex$ be the lexicographic term order induced by the natural order of the variables:
\[ 
p_{u,i}>p_{v,m}\quad\text{iff}\quad u<v,\text{\ or\ } u=v \ {\rm and\ } i<m.
 \]
\begin{Theorem}\label{thm:Ollie}
  \begin{enumerate}
  \item The set $G(I_0)$ forms a Gr\"obner basis for $I_0$ with respect to
    $\plex$.
  \item For any $S\in\ML$, the generators of $I_{S}$ form a Gr\"obner basis with respect to $\plex$.
  \end{enumerate}
\end{Theorem}
\begin{Theorem}\label{thm:J0}
  The ideals $I_0$ and $I_S$ are all prime.
\end{Theorem}

\begin{Theorem}\label{thm:intersection}
  A minimal primary decomposition of the radical of $\ID$ is given by
  \begin{equation*}
    \sqrt{\ID}= I_0 \cap
    \bigcap_{S \in \LL} I_S.
  \end{equation*}
\end{Theorem}

\begin{Corollary}\label{cor:number}
  The number of minimal prime components of $\sqrt{\ID}$ is ${\ell}^{k} - \ell + 1$.
\end{Corollary}
\begin{Proposition}\label{cor:dim}
  The dimension of $R / I_S$ for $S \in \ML$ is $\ell(k + d - 1) - k$. The dimension of $R / I_0$ is $\ell (k + d) - d - 1$.
\end{Proposition}

The proof of these statements are in Section~\ref{sec:proofs}.  We conclude this section with examples and a comparison of our result to similar results about determinantal hyperedge ideals and the CI ideals of the intersection axiom.

\begin{Example}
  \label{ex:s2td3}
  Let us consider some specific cases.
  \begin{enumerate}
  \item Let $k = 2, \ell = 3, s = 2$ and $t = d = 3$. We continue Example~\ref{example:k2l3s2td3}
    and recall,
    $$\Delta = \{12, 34, 56, 135, 246\},\quad\text{and}\quad P= 
  \begin{pmatrix}
    p_{1,1}     &p_{1,2}    &p_{1,3}   &p_{1,4}    &p_{1,5}    &p_{1,6}      \\
    p_{2,1}     &p_{2,2}    &p_{2,3}   &p_{2,4}    &p_{2,5}    &p_{2,6}     \\
    p_{3,1}     &p_{3,2}    &p_{3,3}   &p_{3,4}    &p_{3,5}    &p_{3,6}       \\
\end{pmatrix}.  
    $$
   The corresponding hyperedge ideal is
    $$J_{[3],\Delta} =\left\la [12|12],[13|12],[23|12],[12|34],[13|34],[23|34],[12|56],[13|56],[23|56],[135],[246] \right\ra.$$
    $J_{[3],\Delta}$ has $7$ prime components corresponding to the sets: $14, 16, 32, 36, 52$ and $54$ along with the ideal $I_0$.
  \item 
    Now we increase $k$ to~$3$.  Then
    \begin{gather*}
      \MY= 
      \begin{pmatrix}
        1   &4  &7  \\
        2   &5  &8  \\
        3   &6  &9  
      \end{pmatrix}, \qquad
      \Delta = \{12, 13, 23, 45, 46, 56, 78, 79, 89, 147, 258, 369\}, \\
      P= 
   \begin{pmatrix}
    p_{1,1}     &p_{1,2}    &p_{1,3}   &p_{1,4}    &p_{1,5}    &p_{1,6}    &p_{1,7}    &p_{1,8}  &p_{1,9}    \\
    p_{2,1}     &p_{2,2}    &p_{2,3}   &p_{2,4}    &p_{2,5}    &p_{2,6}    &p_{2,7}    &p_{2,8}  &p_{2,9}   \\
    p_{3,1}     &p_{3,2}    &p_{3,3}   &p_{3,4}    &p_{3,5}    &p_{3,6}    &p_{3,7}    &p_{3,8}  &p_{3,9}  \\
\end{pmatrix}.
    \end{gather*}
    $J_{[3],\Delta}$ has $25$ prime components which fall into three symmetry classes:
    \smallskip
    \begin{center}
      \begin{tabular}{ccccc}
        \toprule
        Type    & Representative      &Occurrences   &Number of generators     &Codimension \\
        \midrule
        0       &$I_0$        &1            &54             &14     \\
        1       &$I_{159}$    &6            &18             &12     \\
        2       &$I_{126}$    &18           &21             &12     \\
        \bottomrule
      \end{tabular}
    \end{center}
  \end{enumerate}
\end{Example}

\begin{Remark}
The ideals $\ID$ do not behave similarly to the determinantal facet ideals studied in \cite{Fatemeh}. For example, the Betti numbers of $\ID$ and $I_0$ are not equal to those of their initial ideals.  This is in contrast to determinantal facet ideals, see Corollary~5.6 in \cite{Fatemeh2}.
\end{Remark}

\begin{Remark}[Statistical interpretation of prime components]
The component~$I_{0}$ describes all joint distributions with full support (i.e.\ without zero entries).  Algebraically, this is equivalent to $I_{0} = \sqrt{\ID} : (\prod_{i,j}p_{i,j})^{\infty}$.  For CI ideals, the component(s) that describe the probability distributions with full support are of special importance~\cite{Sturmfels02:Solving_polynomial_equations}, because it corresponds to the statistical case without structural zeros.
%
$I_{0}$ has a clear probabilistic interpretation:
   $I_{0} = J_{\ind{X}{Y_{1}}[Y_{2}], \ind{X}{\{Y_{1},Y_{2}\}}[H_{2}]},$
  where $|\MH_{2}|=2$.  This can be compared with the intersection axiom, which states:
  \emph{If $\ind{X}{Y_{1}}[Y_{2}]$ and $\ind{X}{Y_{2}}[Y_{1}]$ and if the joint distribution of $X$, $Y_{1}$ and $Y_{2}$ has full support, then $\ind{X}{\{Y_{1},Y_{2}\}}$.}%

  In all other prime components, the zeros of the joint distribution matrix are such that each determinant implied by the CI statement $\ind{X}{Y_{1}}[\{Y_{2},H_{2}\}]$ involves a zero column.  Thus, each such determinant vanishes trivially, and only the determinants implied by $\ind{X}{Y_{1}}[Y_{2}]$ prevail as generators.
\end{Remark}
\section{Proofs of main results}
\label{sec:proofs}

\begin{proof}[{\bf{Proof of Theorem~\ref{thm:Ollie}}}]
  Statement (2) is proved as follows. $I_{S}$ is a sum of a monomial ideal and determinantal ideals.  Each determinantal ideal corresponds to a column $C_{j}$, $j\in[\ell]$, and is generated by all $2$-minors of~$P_{C_{j}\setminus S}$.  The 2-minors form a Gr\"obner basis of this determinantal ideal, see \cite{Fatemeh2}.  Since the sets of variables that appear in the generators of each of the ideals in the sum are disjoint, the generators form a Gr\"obner basis for~$I_S$.
  
  We prove Statement~(1)
using Buchberger's algorithm. Take two elements $g_1$ and $g_2$ in $G(I_0)$. We show that the $S$-polynomial $S(g_1,g_2)$ reduces to zero.  Assuming that the initial terms of $g_1$ and $g_2$ are not coprime, we consider the following cases:

\smallskip

\textbf{Case 1.} $\deg(g_1) = \deg(g_2) = 2$. Since the initials terms of $g_1$ and $g_2$ are not coprime, we deduce that $g_1$ and $g_2$ belong to a submatrix $P_{C_j}$ for some~$j$.
It is a classical result that the $2$-minors of $P_{C_j}$ form a Gr\"obner basis with respect to $\plex$ for the ideal they generate.  Thus, $S(g_1,g_2)$ reduces to zero via the same reduction as in the classical case. 

\smallskip

\textbf{Case 2.} $\deg(g_1) = \deg(g_2) = t$. We use the same classical result, that is, the set of $t$-minors of $P$ is a Gr\"obner basis for the ideal they generate. So the $S$-polynomial has a reduction to zero by $t$-minors. Note that every $t$-minor either occurs in $G(I_0)$ or can be written as a $\KK[p_{i,j}]$-linear sum of the 2-minors in $G(I_0)$. So the reduction of the $S$-polynomial $S(g_1, g_2)$ can be written in terms of minors in $G(I_0)$.

\smallskip

\textbf{Case 3.} $\deg(g_1) = 2$ and $\deg(g_2) = t$. Let us write $g_1 = [i_1 i_2 \mid j_1 j_2]$ and $g_2 = [I \mid J]$.
Take the submatrix $\tilde{P}$ of $P$ with columns $J \cup \{j_1, j_2\}$ and rows $I \cup \{i_1, i_2\}$. Assume that $n = \lv \{i_1, i_2 \} \cup I\rv $ and $m = \lv \{j_1, j_2 \} \cup J \rv$. We relabel the rows of $\tilde{P}$ with $[n]$ and its columns with $[m]$.
The indices $j_1, j_2$ belong to the column ${C_i}$ for some $i$, hence by the definition of $G(I_0)$ there is at most one $j \in J$ with $j \in {C_i}$.  Denote by $j_u$ the unique element of $\{j_1, j_2 \} \cap J$. Define $\I = \{i_1, i_2 \} \cap I$. We have that $\lv \I \rv \le 2$.

\smallskip

We now study all four possible cases for $u \in \{1, 2\}$ and $\lv \I \rv \in \{ 1, 2\}$ depicted in Figure~\ref{fig:thm_Ollie_cases}.

\smallskip

\textbf{Case 3.i.} $u = 1$ and $\lv \I \rv = 1$.
We have $i_1 = j_1$, because $p_{i_1, j_1}$ divides the initial term of~$g_2$. Let $a := i_1 = j_1$ and $b := i_2 > a$. Since every element $j \in J$ belongs to a distinct column $C_r$ and the minor $g_1$ is taken from a single column $C_i$, we must have that $j_1, j_2$ are adjacent columns in $\tilde{P}$ (recall the ordering of the set $\MY$ of column indices defined in Section~\ref{sec:determ-hyper-ideals}).
Thus, $j_2 = a+1$. Note that in this case $m = n$.
The following identity can be deduced by expanding the determinants on the left hand side along the columns $a+1$ and $a$, respectively:
\begin{multline} \label{eqn:spoly1}
    p_{b,a}[[n] \backslash \{b\} \mid [n] \backslash \{a\}] - p_{b,a+1}[[n] \backslash \{b\} \mid [n] \backslash \{a+1\}] \\
    = \sum_{i < b} (-1)^{a+i-1}[ib \mid a(a+1)][[n] \backslash \{i,b\} \mid [n] \backslash \{a,a+1\}] \\
    + \sum_{i > b} (-1)^{a+i-1}[bi \mid a(a+1)][[n] \backslash \{b,i\} \mid [n] \backslash \{a,a+1\}]
\end{multline}

We now check that rearranging the above equation gives a reduction of $S(g_1, g_2)$ to zero. First, the coefficients of $g_1$ and $g_2$ agree with those from the $S$-polynomial. Second, the initial terms of all other terms appearing in the expression are smaller than the leading term of the $S$-polynomial.

Note that $g_2 = [[n] \backslash \{b\} \mid [n] \backslash \{a+1\}]$ appears on the left hand side of~\eqref{eqn:spoly1} with coefficient $-p_{b,a+1}$ as expected. On the other hand, $g_1 = [ab \mid a(a+1)]$ appears on the right hand side 
with coefficient $(-1)^{2a-1} \init(h)$ where $h = [[n] \backslash \{a,b\} \mid [n] \backslash \{a,a+1\}]$, as required.

It remains to check that all terms appearing in \eqref{eqn:spoly1} which do not belong to the $S$-polynomial, are less than the initial term of $S(g_{1},g_{2})$. To do this we use the table in Appendix~\ref{tab:ini-3i}. First we calculate explicitly the initial term of the $S(g_1, g_2)$ which is either the second largest term of $\init(h)g_1$ or $p_{b,a+1}g_2 $ where $h = [[n] \backslash \{a,b\} \mid [n] \backslash \{a,a+1\}]$.  The initial term of $S(g_1, g_2)$ depends upon the values of $a$ and $b$ so we take the necessary cases in the table. Second we see from the table that every other term in \eqref{eqn:spoly1} is less than or equal to the initial term of $S(g_1, g_2)$.

\smallskip

\textbf{Case 3.ii.} $u = 2$ and  $\lv \I \rv = 1$.
We have $i_2 = j_2$, because $p_{i_2, j_2}$ divides the initial term of $g_2$. Let $a + 1 := i_2 = j_2$ and $b := i_1 < a + 1$. By the relabelling to $\tilde{P}$ we have that $j_1 = a$. Note that in this case $m = n$.

We now check that rearranging \eqref{eqn:spoly1} from Case 3.i gives a reduction of $S(g_1, g_2)$ to zero. Note that the value of $b$ is different in the current case to Case 3.i; however, the same relation holds by the same proof.

Note that $g_2 = [[n] \backslash \{b\} \mid [n] \backslash \{a\}]$ appears on the left hand side of~\eqref{eqn:spoly1} with coefficient $p_{b,a}$ as expected. On the other hand, $g_1 = [b(a+1) \mid a(a+1)]$ appears on the right hand side with coefficient $(-1)^{2a} \init(h)$ where $h = [[n] \backslash \{b,a+1\} \mid [n] \backslash \{a,a+1\}]$, as required.

Next we check that all other terms appearing in \eqref{eqn:spoly1} are smaller than the initial term of $S(g_1,g_2)$. We do this by referring to the table in Appendix~\ref{tab:ini-3ii}. First we calculate explicitly the initial term of the $S(g_1, g_2)$ which is either the second largest term of $\init(h)g_1$ or $p_{b,a}g_2 $ where $h = [[n] \backslash \{b, a+1\} \mid [n] \backslash \{a,a+1\}]$. Second we see from the table that every other term in \eqref{eqn:spoly1} is less than or equal to the initial term of $S(g_1, g_2)$.

\smallskip

\textbf{Case 3.iii.} $u= 1$ and  $\lv \I \rv = 2$.
We have $i_1 = j_1$, because $p_{i_1, j_1}$ divides the initial term of $g_2$. Let $a := i_1 = j_1$ and $b := i_2 > a$. By the relabelling to $\tilde{P}$ we have that $j_2 = a+1$. Note that in this case $m = n+1$.
The following identity can be deduced by expanding the determinants on the left hand side along the columns $a+1$ and $a$, respectively:
\begin{multline} \label{eqn:spoly3}
    p_{b,a}[[n] \mid [m] \backslash \{a\}] - p_{b,a+1}[[n] \mid [m] \backslash \{a+1\}] \\
    = \sum_{i < b} (-1)^{a+i-1}[ib \mid a(a+1)][[n] \backslash \{i\} \mid [m] \backslash \{a,a+1\}] \\
    + \sum_{i > b} (-1)^{a+i}[bi \mid a(a+1)][[n] \backslash \{i\} \mid [m] \backslash \{a,a+1\}]
\end{multline}

We now check that rearranging the above equation gives a reduction of $S(g_1, g_2)$ to zero.
%
Note that $g_2 = [[n] \mid [m] \backslash \{a+1\}]$ appears on the left hand side of~\eqref{eqn:spoly3} with coefficient $p_{b,a+1}$ as expected. On the other hand, $g_1 = [ab \mid a(a+1)]$ appears on the right hand side with coefficient $(-1)^{2a} \init(h)$ where $h = [[n] \backslash \{a\} \mid [m] \backslash \{a,a+1\}]$, as required.

Next we check that all other terms appearing in \eqref{eqn:spoly3} are smaller than the leading term of $S(g_1,g_2)$. This can be seen from the table in Appendix~\ref{tab:ini-3iii}. First we calculate explicitly the initial term of the $S(g_1, g_2)$ which is either the second largest term of $\init(h)g_1$ or $p_{b,a+1}g_2 $ where $h = [[n] \backslash \{a\} \mid [m] \backslash \{a,a+1\}]$. Second we see from the table that every other term in \eqref{eqn:spoly1} is less than or equal to the initial term of $S(g_1, g_2)$.

\smallskip

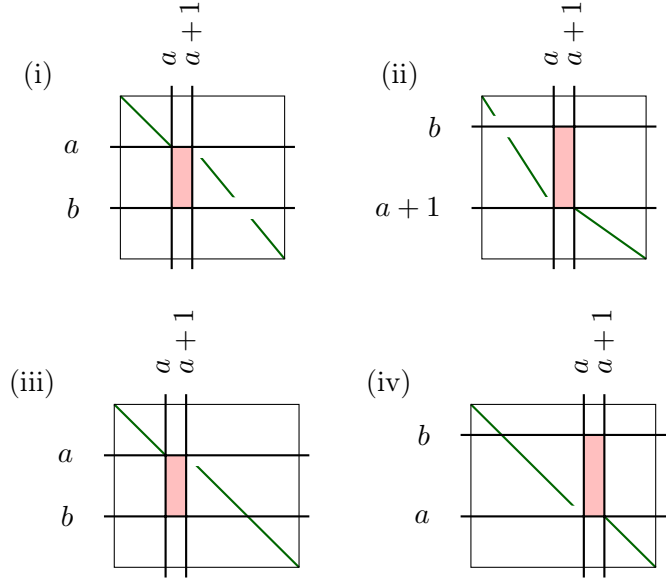
\begin{figure}
  \begin{center}
    \newcommand{\tikzPtilde}[8]{
      \node[label=left:$#1$] (ar) at ($(-1,8) - (0,#5)$) {};
      \node (ar') at ($(ar -| 9,0) + (#8,0)$) {};
      \node[label=left:$#2$] (b) at ($(-1,8) - (0,#6)$) {};
      \node (b') at ($(b -| 9,0) + (#8,0)$) {};
      \node[label={[rotate=90,text depth=2ex]right:$#3$}] (ac) at ($(0,9) + (#7,0)$) {};
      \node (ac') at (ac |- 0,-1) {};
      \node[label={[rotate=90,text depth=2ex]right:$#4$}] (a1) at ($(1,9) + (#7,0)$) {};
      \node (a1') at (a1 |- 0,-1) {};

      \draw[black!60!green,thick] (0,8) -- (ar -| ac)
                                  (ar -| a1) -- ($(8,0) + (#8,0)$);

      \draw[fill=white,color=white] ($(b -| 0,0) - (0,0.5) + #8*(0,0.5)$) rectangle ($(b -| 8,0) + (0,0.5) - #8*(0,0.5)$);
      \draw[fill=white,color=white] ($(ar -| 0,0)!0.5cm!(b -| 0,0)$) rectangle ($(ar -| 8,0) + (#8,0)$);

      \fill[white!75!red] (ar -| ac) rectangle (b -| a1);
      \draw[thick] (ar) -- (ar')
                   (b) -- (b')
                   (ac) -- (ac')
                   (a1) -- (a1');
      \draw (0,8) rectangle ($(8,0) + (#8,0)$);
    }
    \begin{tabular}{cc}
      \begin{tikzpicture}[scale=0.27]
        \node at (-4,9) {(i)};
        \tikzPtilde{a}{b}{a}{a+1}{2.5}{5.5}{2.5}{0}
      \end{tikzpicture}
      &
      \begin{tikzpicture}[scale=0.27]
        \node at (-4,9) {(ii)};
        \tikzPtilde{a+1}{b}{a}{a+1}{5.5}{1.5}{3.5}{0}
      \end{tikzpicture}
      \\
      \begin{tikzpicture}[scale=0.27]
        \node at (-4,9) {(iii)};
        \tikzPtilde{a}{b}{a}{a+1}{2.5}{5.5}{2.5}{1}
      \end{tikzpicture}
      &
      \begin{tikzpicture}[scale=0.27]
        \node at (-4,9) {(iv)};
        \tikzPtilde{a}{b}{a}{a+1}{5.5}{1.5}{5.5}{1}
      \end{tikzpicture}
    \end{tabular}
  \end{center}
  
  \caption{The matrix $\tilde{P}$ in Cases 3.i -- 3.iv. Vertical and horizontal lines represent columns and rows of $\tilde{P}$. The shaded rectangle represents $g_1 = [x,b \mid a, a+1]$, where $x \in \{ a,a+1\}$ depends on the case. The diagonal lines represent the leading term of $g_2$. So $g_2 = [I|J]$ is the minor of the submatrix obtained by removing vertical and horizontal lines which do not meet any of the diagonal lines. The value $|\MI|$ is the number of horizontal lines bounding the shaded rectangle which meet the diagonal line. So $|\MI| = 1$ in Cases 3.i and 3.ii and $|\MI| = 2$ in Cases 3.iii and 3.iv. If the diagonal line meets the top left corner of the shaded rectangle, then $u = 1$ as in Case 3.i and Case 3.iii. If the diagonal line meets the bottom right corner of the shaded rectangle then $u = 2$ as in Case 3.ii and Case 3.iv. In Cases 3.i and 3.ii, $\tilde{P}$ is a square matrix and $g_2$ is obtained by removing a row and column. In Cases 3.iii and 3.iv, $\tilde{P}$ has one more column than it has rows and $g_2$ is obtained by removing a single column.}
  \label{fig:thm_Ollie_cases}
\end{figure}
\textbf{Case 3.iv.} $u = 2$ and $\lv \I \rv = 2$.
We have $i_2 = j_1$, because $p_{i_2, j_2}$ divides the initial term of $g_2$ which is obtained by taking the determinant of $\tilde{P}$ after removing column $j_1$. Let $a := i_2 = j_1$ and $b := i_1 < a$. By the relabelling to $\tilde{P}$ we have that $j_2 = a+1$. Note that in this case $m = n+1$.

We now check that rearranging \eqref{eqn:spoly3} from Case 3.iii gives a reduction of $S(g_1, g_2)$ to zero. Note that the value of $b$ is different in the current case to Case 3.iii however the same equation holds by the same proof.

Note that $g_2 = [[n] \mid [m] \backslash \{a\}]$ appears on the left hand side of~\eqref{eqn:spoly3} with coefficient $p_{b,a}$ as expected. On the other hand, $g_1 = [ba \mid a(a+1)]$ appears on the right hand side with coefficient $(-1)^{2a} \init(h)$ where $h = [[n] \backslash \{a\} \mid [m] \backslash \{a,a+1\}]$, as required.

Next we check that all other terms appearing in \eqref{eqn:spoly3} are smaller than the initial term of $S(g_1,g_2)$. This can be seen from the table in Appendix~\ref{tab:ini-3iv}. First we calculate explicitly the initial term of the $S(g_1, g_2)$ which is either the second largest term of $\init(h)g_1$ or $p_{b,a}g_2 $ where $h = [[n] \backslash \{a\} \mid [m] \backslash \{a,a+1\}]$. Second we see from the table that every other term in \eqref{eqn:spoly1} is less than or equal to the initial term of $S(g_1, g_2)$.
%
%
%
%
%
\end{proof}

\begin{Proposition} \label{prop:intersection}
  Let $d = \ell = t$ and $ s = 2$. Then,
  \[
    \sqrt{\ID}= I_0 \cap \bigcap_{S \in\ML} I_S,
    \qquad \text{with $\ML$ as in Theorem~\ref{thm:intersection}.}
  \]
\end{Proposition}
\begin{proof}
  Let $J = \bigcap_{S \in \LL} I_S$.
  By Theorem~\ref{thm:Ollie} the ideals $I_S$ and $I_0$ have a squarefree Gr\"obner basis, therefore each ideal is radical. Also
   $\ID \subseteq I_S$ for each $S\in\ML$, hence $\sqrt{\ID} \subseteq J \cap I_0$. For the opposite inclusion, it is sufficient to show that $V(\ID) \subseteq V(J) \cup V(I_0)$.
So take
$$\bolda = 
\begin{pmatrix}
a_{1,1}		    &a_{1,2}	&\cdots		&a_{1,k\ell}	\\
a_{2,1}		    &a_{2,2}	&\cdots		&a_{2,k\ell}	\\
\vdots		    &\vdots		&			&\vdots		\\
a_{\ell,1}		&a_{\ell,2}	&\cdots		&a_{\ell,k\ell}	\\
\end{pmatrix}\in V(\ID)\setminus V(J).$$
It is enough to show that $\bolda \in V(I_0)$.  Since $\bolda \notin \bigcup_{S \in \LL} V(I_S)$, the submatrix
$\bolda_S$ is non-zero for every $S \in \LL$.  Now, we show that $\lv \bolda_B \rv = 0$ for any $B\subset[k\ell]$ with
$|B|=\ell$.

Assume first that there exist $i,i'$ such that $\MY_{i,j}, \MY_{i',j}\in B$. We write $b = \MY_{i,j}$ and $b' = \MY_{i',j}$.
Since $b, b'$ belong to the same column $C_j$ of $\MY$, all 2-minors in the submarix $\bolda_{\{b, b'\}}$ are $0$.  When expanding $\lv \bolda_B \rv$ along all columns except~$b, b'$, the result is an expression for $\lv \bolda_B \rv$ in terms of the 2-minors in the submatrix~$\bolda_{\{b, b'\}}$, and so $\lv \bolda_B \rv = 0$.

Otherwise, if such $i,i'$ do not exist, then $B = \{a_1,\dots,a_\ell\}$ with $a_j \in C_j$.
For each $i$ we let $r_i = \lv B \cap R_i \rv$.
Let $r_B = \max_i\{r_i\} = \max_i\{\lv B \cap R_i \rv \}$.  Choose $i_B$ such that $r_{i_B} = r_{B}$.

We proceed by reverse induction on~$r_B$.
For the base case, suppose $r_B = \ell$.  That is $B = R_{i_B}$.
By definition, the minor $[B]$ is a generator of $\ID$, hence $\lv \bolda_B \rv = 0$.

Now let us suppose that $0 < r_B < \ell$. We write $B = \{a_1,a_2,\dots,a_{\ell} : a_j \in C_j\}$. Let us assume by contradiction that $\lv \bolda_B \rv \neq 0$ i.e.\ $\matrank(\bolda_B) = \ell$. Define $T := R_{i_B} \backslash B$ which is non-empty because $r_B  < \ell$. Let $\tau \in T$. Then $\tau \in C_j$ for some $j$. Suppose that column $\bolda_{\{\tau\}}$ is not a zero column. Let
$B' = (B \cup \{\tau\}) \backslash \{a_j\}$.  Then $r_{B'} = r_{B} + 1$.
Since all 2-minors in $\bolda_{\{\tau,a_j\}}$ are $0$ it follows that $\bolda_{\{\tau\}}, \bolda_{\{a_j\}}$ are linearly dependent. Since they are both non-zero we deduce that $\matrank(\bolda_B) = \matrank(\bolda_{B'})$. So by the inductive hypothesis  $\matrank(\bolda_{B'}) < \ell$, and we have a contradiction. Therefore, $\bolda_{\{ \tau \}} = \underline{0}$ is a zero column for each $\tau \in T$. 

Fix some $b = \MY_{i_{B},j} \in T$. Since $r_B < \ell$ there exists $R_{i'} \neq R_{i_B}$ for which $R_{i'} \cap B \neq \emptyset$. Let,
$$T' := \{\MY_{i', \beta} : \beta \in [\ell], a_{\beta} \in R_{i_B} \} \neq \emptyset.$$
Suppose there exists $b' = \MY_{i',j'} \in T'$ such that $\bolda_{\{b'\}} = \underline{0}$. If for each $i'' \in [k] \backslash \{i_{B},i'\}$ there exists $j''_{i''} \in [\ell]$ with $\bolda_{\{\MY_{i'',j''_{i''}}\}} = \underline{0}$ then the submatrix $\bolda_{S}$ of $\bolda$ with columns $S=\{b, b', \MY_{i'', j''_{i''}}\}\in\LL$ is the zero matrix contradicting our assumption $\bolda \not\in \bigcup_{S \in \LL} V(I_S)$.
Hence for some $i'' \in [k]$, each column of $\bolda$ indexed by $\MY_{i'',1}, \dots, \MY_{i'',\ell}$ is non-zero.
 For each $j'' \in [\ell]$, the columns $\bolda_{\{\MY_{i'',j''}\}}$ and $\bolda_{\{a_{j''}\}}$ are linearly dependent and non-zero.  We deduce that $\matrank(\bolda_B) = \matrank(\bolda_{R_{i''}}) < \ell$, a contradiction. Therefore for each $b' = \MY_{i',j'} \in T'$ we must have $\bolda_{\{b'\}} \neq \underline{0}$. 

Let $B' = (B \cup T') \backslash R_{i_B}$.  Then $r_{B'} = r_{B} + r_{i'} > r_{B}$.
For each $b' = \MY_{i',j'} \in T'$ the columns $\bolda_{\{b'\}}$ and $\bolda_{\{a_{j'}\}}$ are linearly dependent and non-zero. So we deduce that $\matrank(\bolda_B) = \matrank(\bolda_{B'})$. 
By the inductive hypothesis, $\matrank(\bolda_{B'}) < \ell$, a contradiction. So $\matrank(\bolda_B) < \ell$ and we have shown that $\lv \bolda_B \rv = 0$. 

Therefore $\bolda \in V(I_0)$, which completes the proof.
\end{proof}

\begin{proof}[{\bf Proof of Theorem~\ref{thm:intersection}}]
  We show that the statement in Proposition~\ref{prop:intersection} holds for $d > \ell$.
  As before, let $\bolda = (a_{i,j}) \in V(\ID)$ and assume that $\bolda \not\in V \left( \bigcap_{S \in \LL} I_S \right) = \bigcup_{S \in \LL} V\left(I_S \right)$.
  We may assume that $K$ is infinite.  If not, then replace $K$ by an infinite algebraic extension~$K'$ (e.g. its algebraic closure).  The matrix $\bolda$ can also be interpreted as a matrix over~$K'$, which does not change its minors.

  The statement of the theorem is invariant under multiplication from the left by elements of~$\operatorname{GL}_{d}(K)$, in the following sense: Let $G\in\operatorname{GL}_{d}(K)$, and let $T\subseteq\MY$.  Then $G$ induces a bijection of
  \begin{equation*}
    \la[S|T] : S\subseteq[d], |S| = |T|\ra
  \end{equation*}
  that is linear on the generators.
  By assumption, for any $S\in\ML$, there exists $j_{S}\in S$ such that the $j_{S}$-th column $\bolda_{j_{S}}$ of $\bolda$ does not vanish.
  Applying a suitable coordinate transformation in $\operatorname{GL}_{d}(K)$, we may assume that the entries
  $a_{1,j_{S}}\neq 0$ for all~$S\in\ML$ (this is possible since $K$ is infinite by assumption).

  Let $A\subseteq\MX$ with $|A|=\ell$.  If $1\in A$ (or, more generally, if for any $S\in\ML$ there exist $i\in A$, $j\in\MY$ with $a_{i,j}\neq 0$), then $[A|T](\bolda)=0$ by Proposition~\ref{prop:intersection} applied to the submatrix of $\bolda$ of those rows indexed by~$A$.  Otherwise, let $A'=A\cup\{1\}$, and let $\bolda'$ be the submatrix of $\bolda$ obtained by restricting to the rows indexed by~$A'$.
  We consider two cases:

  First, assume that $T$ contains $t_{0}$ with $a_{1,t_{0}}\neq 0$ (for example, $t_{0}:= j_{A'}$ for some $A'\in\ML$.
  Let $\boldb$ be the matrix obtained from $\bolda[A|T]$ by adding a copy of the $t_{0}$th column to the end.  Then $|\boldb|=0$.
  Let $T=\{t_{1},\dots,t_{\ell}\}$.  Expanding $|\boldb|$ along the last column gives
  \begin{equation*}
    a_{1,t_{0}}|\bolda| = \sum_{a\in A}(\pm1)a_{a,t_{0}}[A\cup\{1\}\setminus\{a\}|T](\bolda).
  \end{equation*}
  Here, $a_{1,t_{0}}\neq 0$, and $[A\cup\{1\}\setminus\{a\}|T](\bolda) = 0$ by Proposition~\ref{prop:intersection}.

  Second, assume that $a_{1,t}=0$ for all~$t\in T$.  Let $t_{0}\in\MY$ be arbitrary with $a_{1,t_{0}}\neq 0$,
  and let $\boldb$ be the matrix obtained from $\bolda[A|T]$ by adding a copy of the $t_{0}$th column to the end.  Let $T=\{t_{1},\dots,t_{\ell}\}$.
  Expanding $|\boldb|$ along the first row gives $|\boldb|= a_{1,t}[A|T]$.
  Expanding $|\boldb|$ along the first column shows that $|\boldb|$ is a linear combination of minors of the form $[A''|T'']$ with
  $A''\subset A'$, $T'' = T\cup\{t_{0}\}\setminus\{t_{1}\}$ and $|A''|=|T''|=t$.  These minors $[A''|T'']$ vanish by the first case.  Thus, $[A|T](\bolda) = a_{1,t}^{-1}|\boldb| = 0$.
\end{proof}

\begin{proof}[{\bf Proof of Corollary~\ref{cor:number}}]
  We first show that for each pair of distinct subsets $S,T$ in $\mathcal{L}\cup \{\emptyset \}$ neither $I_S\subset I_T$ nor $I_T\subset I_S$. 

  Suppose by contradiction that $I_S \subset I_T$ for some $S, T \in \LL$. Then $I_T$ contains all the variables in $I_S$. These variables are exactly those indexed by $S$. However the only variables contained in $I_T$ are those indexed by $T$. Since $\lvert S \rvert = \lvert T \rvert$ we deduce $S = T$, a contradiction.

  Next we take $S \in \LL$ and show that $I_S \not\subset I_0$ and $I_0 \not\subset I_S$. Note that $I_0$ contains no variables which is easily checked by applying Theorem~\ref{thm:Ollie}. Since $I_S$ contains variables indexed by $S$, we conclude that $I_S \not\subset I_0$. For the other non-inclusion consider the set
  $$A = R_1 \cup \{d\} \backslash S $$ 
  where $R_1$ is the first row of $\MY$, $d \in C_i \backslash S$ and $C_i$ is the unique column for
  which $R_1 \cap S \subset C_i$. Note that $C_i \backslash S \neq \emptyset$ by the
  definition of $\LL$. Clearly $[A] \in I_0$ since $I_0$ contains all maximal minors of $P$. Next we
  prove that $[A] \not\in I_S$ from which it follows that $I_0 \not\subset I_S$.

  By Theorem~\ref{thm:Ollie}, the set of $2$-minors and variables that generate $I_S$ form a
  Gr\"obner basis for $I_S$ with respect to~$\plex$.
  Then $[A] \not\in I_S$ follows from the fact that for each generator $g$ of $I_S$, $\init(g)$ does
  not divide $\init([A])$.

  By Theorem~\ref{thm:intersection}, the number of prime
  components of $\ID$ is equal to the size of
  $\mathcal{L} \cup \{\emptyset \}$,
  which is ${\ell}^k - \ell + 1$.
\end{proof}

\begin{proof}[{\bf{Proof of Proposition~\ref{cor:dim}}}]
First let $S \in \ML$. The generating set
of $I_S$ is a Gr\"obner basis with respect to $\prec_{\lex}$. Note that $\init(g)$ is squarefree for each generator $g \in G$.  Let $\Psi$ be the simplicial complex on $\MX \times \MY$ such that the Stanley-Reisner ideal of $\Psi$ is $\init(I_S)$.
Note, 
\[\dim(R/I_S) = \dim(R/\init(I_S)) = \dim(\Psi) + 1 = \max_{f \in\Psi}\{|f|\}.\]

So it suffices to show that facets of $\Psi$ of maximum size have cardinality $\ell(d+k-1) -k$.  Thus, we first construct a face $f$ of $\Psi$ of size $\ell(d+k-1) -k$, and then we show that no face has size larger than $\ell(d+k-1) -k$.

Let $c_i = \max\{C_i \backslash S\}$. We construct $f$ as the disjoint union of $f_{1}$ and $f_{2}$, where
\[f_{1} = \{(d,t) : t \in [\ell k] \backslash S \}
  \quad\text{ and }\quad
  f_{2} = \{(r,c_i) : 1 \le r \le d-1, 1 \le i \le \ell \}. \] Note that $C_i \backslash S$ is non-empty because $S \in \ML$. We have that 
$|f_{1}| = \ell k - k$ since $|S| = k$ and 
$ |f_{2}| = \ell (d - 1)$. These sets are disjoint, so $|f| = \ell(d + k - 1) -k$.
It is straightforward to check that $f$ is indeed a face of~$\Psi$.

Now let $f'$ be a face of $\Psi$. We show that $|f'| \le |f|$. For each $i \in [\ell]$ consider the submatrix $B_i = [d] \times (C_i \backslash S) \subseteq \MX \times \MY$. Note that $f' \cap ([d]\times S) = \emptyset$ because each $(i,j) \in [d]\times S$ is a minimal non-face of $\Psi$. 
For each quadratic generator of $I_S$ whose variables are indexed by elements of $B_i$, we obtain a minimal non-face of $\Psi$. By observing these non-faces it is straightforward to show $|f' \cap B_i| \le (d - 1) + |C_i \backslash S|$. Summing over the $B_i$'s gives,
\[|f'| = \sum_{i = 1}^{\ell} |f' \cap B_i| \le \ell(d - 1) + \sum_{i = 1}^{\ell} |C_i \backslash S| = \ell(d-1) + |[k\ell] \backslash S|= \ell(k + d - 1) - k.\]

The proof of $\dim(I_0) = \ell(k + d) - d - 1$ follows similarly. Let $\Psi$ be the simplicial complex on the vertex set $\MX \times \MY$ associated to the Stanley-Reisner ring $R/\init(I_0)$.  Let us construct $f$ as a disjoint union of $f_1, f_2$ and $f_3$, where
\[f_1 = \{(1,t) : t \in [k\ell] \}, \ 
f_2 = \{(r,j) : 2 \le r \le d, j \in (R_1 \backslash C_{\ell}) \} \ \text{ and } \]
\[
f_3 = \{ (r,k(\ell -1) + 1) : 2 \le r \le \ell - 1 \}. \]
It is straightforward to see that $f$ does not contain any minimial non-faces of $\Psi$ hence it is a face of size $\ell(k + d) - d - 1$. To show that $f$ is a facet, we take any subset of vertices $f' \subseteq \MX \times \MY$ and show that if $|f'| = \ell(k + d) - d$ then $f'$ contains a minimal non-face of $\Psi$.

Let $B_i = [d] \times C_i \subset \MX \times \MY$ for each $1 \le i \le \ell$. If  $f' \cap B_i$ contains two elements of the form $(i_1, j_1), (i_2, j_2)$ with $i_1 < i_2$ and $j_1 < j_2$ then $[i_1, i_2 | j_1, j_2]$ 
is a generator of $I_0$ with initial term $p_{i_1, j_1} p_{i_2, j_2}$. Hence $\{(i_1, j_1), (i_2, j_2) \}$ is a non-face of $\Psi$. So we may assume that no pairs exist for any $f' \cap B_i$. In general we say a subset $A \subset \MX \times \MY$ satisfies condition $(*)$ if for each $A \cap B_i$, there are no pairs $(i_1, j_1), (i_2, j_2) \in A \cap B_i$ with $i_1 < i_2$ and $j_1 < j_2$. It follows that $|f' \cap B_i| \le k + d - 1$ for all $i$.

For each $f' \cap B_i$ with size strictly less than $k + d - 1$, one can always find an element $(a,b) \in B_i$ such that $(f' \cap B_i) \cup \{(a,b)\}$ satisfies condition $(*)$. So there exists $f''$ such that $f' \subseteq f'' \subset \MX \times \MY$, $f''$ satisfies $(*)$  and $|f''| = \ell(k + d -1) = |f'| + (d - \ell)$. Since $|f'' \cap B_i| = k + d - 1$, it follows that there is at least one element of $f'' \cap B_i$ belonging to each row. That is, for each $1 \le i \le \ell$ and $1 \le r \le d$ there exists $(r, j_i) \in f'' \cap B_i$ for some $j_i$. Now let us consider $f'' \backslash f'$, which is set of size $d - \ell$. The elements of $f'' \backslash f'$ belong to at most $d-\ell$ distinct rows of $\MX \times \MY$. So there are at least $\ell$ other rows, and we consider any $\ell$-subset $r_1 < \dots < r_\ell \subset \MX$ of these rows.
So we have $(r_1, j_1), (r_2, j_2), \dots, (r_\ell, j_\ell) \in f'$. By construction we have $j_i \in C_i$ for each $i$ and $r_1 < \dots < r_\ell$. So $(r_1, j_1), \dots, (r_\ell, j_\ell)$ is a non-face of $\MX \times \MY$ because $[r_1, \dots, r_\ell | j_1, \dots, j_\ell]$ is a generator of $I_0$.
\end{proof}

\begin{proof}[{\bf{Proof of Theorem~\ref{thm:J0}}}]
  For each $S \in \LL$, the ideal $I_S$ is prime since it is generated by a collection of variables and 2-minors which arise from distinct columns of $P$.
 
 To show that $I_0$ is prime we proceed by induction on $\ell$. For $\ell = 2$ the result holds by Lemma~\ref{base case}. If $\ell>2$ then by Lemma~\ref{localise ideal} we have that $I_0$ is prime for $(k,\ell,d,s,t)$ if and only if $I_0$ is prime for $(k,\ell-1,d,s,t-1)$. By induction hypothesis $I_0$ is prime for $(k,\ell-1,d,s,t-1)$ which completes the proof.
\end{proof}

Now, we mention the lemmas we used in the proof of Theorem~\ref{thm:J0}.
In the following we write $G = G(I_0)$ and fix $k,t = \ell,s = 2, d > \ell$ along with matrices $\MY$ and $P$. We also denote $ B_j$ for the submatrix $P_{C_j}$ for $1 \le j \le \ell$.

\begin{Lemma}
\label{localise vars}
For each $1 \le j \le \ell-1$ the variable $p_{d, (j-1)k+1}$ is a non-zerodivisor modulo~$I_0$.
\end{Lemma}

\begin{proof}
First we show that for each $1 \le j \le \ell-1$ the variable $p_{d, (j-1)k + 1}$ does not divide $\init(g)$ for any $g \in G$. First note that $G$ only contains $2$-minors and $\ell$-minors. Each $2$-minor is obtained from a submatrix lying entirely within a single block $B_i$ for some $i$. If $p_{d, (j-1)k + 1}$ lies inside this block, that is $i = j$, then it is the bottom left entry and so does not lie on the leading diagonal of any $2 \times 2$ submatrix. Also any $\ell$-minor $g \in G$ is obtained from a submatrix whose last column lies in $B_{\ell}$. So the only indeterminates appearing as a factor of $\init(g)$ of the form $p_{\ell, *}$ lie in $B_{\ell}$. However $p_{d, (j-1)k + 1} \in B_j \neq B_{\ell}$ because $1 \le j \le \ell-1$.

Now fix $1 \le j \le \ell-1$ and suppose $x := p_{d,(j - 1)k + 1}$ is a zerodivisor modulo $I_0$. Then there exists $f \in R \backslash I_0$ such that $xf \in I_0$. Suppose without loss of generality that $\init(f)$ is chosen to be as small as possible with respect to $\prec$. We have the monomial $\init(xf) = x \, \init(f) \in \init(I_0) = \la \init(g): g \in G \ra$. Hence $x \, \init(f) \mid \init(g)$ for some $g \in G$. Since $x \nmid \init(g)$ we have $\init(f) = h \, \init(g)$ for some monomial $h$. Let $\bar{f} = f - hg$, note that $\init(\bar{f}) \prec \init(f)$. But $x\bar{f} = xf - xhg \in I_0$, contradicting minimality of $f$. Hence, $x$ is a non-zerodivisor.
\end{proof}

We first recall Lemma~3.10 from \cite{Fatemeh2} which is a useful tool to localize a determinantal ideal in non-zerodivisor variables. 

\begin{Lemma}[Localization Lemma]
  \label{rules}
  Let $P$ be an $m\times n$-matrix of indeterminates and let $I\subset K[P]$ be an ideal generated by a set $G$ of
  minors.
  Furthermore, let $i_{1},\dots,i_{k}\in[m]$ and $j_{1},\dots,j_{k}\in[n]$.  Assume that for each minor $[a_1,\ldots,a_t|b_1,\ldots, b_t]\in G$ 
 the minors
  $[\alpha_1,\ldots, \alpha_t|b_1,\ldots, b_t]$ also belong to $G$ for all
  $\{\alpha_{1},\ldots,\alpha_{t}\}\subset\{i_{1},\dots,i_{k},a_{1},\ldots,a_{t}\}$,
where $\alpha_{1}<\cdots<\alpha_{t}$.
  Then the localizations $(R/I)_{[i_{1},\dots,i_{k}|j_{1},\dots,j_{k}]}\iso
  (R/J)_{[i_{1},\dots,i_{k}|j_{1},\dots,j_{k}]}$ at the minor $[i_{1},\dots,i_{k}|j_{1},\dots,j_{k}]$ are isomorphic, where $J$ is generated by
\begin{itemize}
\item [(a)] the minors $[a_1,\ldots, a_t|b_1,\ldots, b_t]\in G$ with
  $\{b_{1},\ldots,b_{t}\}\cap\{j_{1},\dots,j_{k}\} = \emptyset$,
\item
  [(b)] 
  the minors $[\alpha_{1},\ldots,\alpha_{t-r}|b_1,\ldots,\hat{b}_{k_{1}},\ldots,\hat{b}_{k_{r}},\ldots, b_t]$ where  
  \begin{itemize}
  \item $[a_1,\ldots, a_t | b_1,\ldots, b_t] \in G$ and 
  \item $\{b_{k_{1}}, \ldots, b_{k-r}\} =\{b_{1},\ldots,b_{t}\}\cap\{j_{1},\ldots,j_{k}\}$ and
  \item  $\alpha_{1},\ldots,\alpha_{t-r}\in\{a_{1},\ldots,a_{t},
  i_{1},\ldots,i_{k}\}$.
  \end{itemize}
\end{itemize}
\end{Lemma}

\begin{Lemma}
\label{localise ideal}
$\left(R / I_0 \right)_{p_{d,1}} \cong \left(R / J \right)_{p_{d,1}}$, where
\begin{align*}
    J &= \left\la p_{i,j} : i \in [d], 1 < j \le k \right\ra \\
      & + \left\la [A \mid B] : A\subset [d], B\subset C_j, |A|=|B|=2,  j \ge 2 \right\ra\\
      & + \left\la [A \mid B] : A\subset[d], \tilde{B}\subset[k\ell]\backslash C_1, |A|=|B|=\ell-1, \lv \tilde{B} \cap C_j \rv \le 1 \right\ra.
\end{align*}
In particular $I_0$ is prime for $(k,\ell,d,s,t)$ if and only if $I_0$ is prime for $(k, \ell-1, d, s, t-1)$.
\end{Lemma}

\begin{proof}
  Let $G(I_0)$ be the minimal generating set of $I_0$ from Lemma~\ref{lemma:Generating}. Since $p_{d,1}$ is a non-zerodivisor modulo $I_0$ by Lemma~\ref{localise vars} we apply the Localization Lemma~\ref{rules} to deduce,
  $$\left(R / I_0 \right)_{p_{d,1}} \cong \left(R / J \right)_{p_{d,1}},$$
  where $J$ is generated by:
  \begin{itemize}
  \item $[A \mid B] \in G$ where $1 \not\in B$,
  \item $[A\setminus\{i\}\mid B\setminus\{1\}]$, where $[A\mid B]\in G$, $1\in B$, $i\in A$.
  \end{itemize}
Thus, the 2-minors $[A \mid B] \in G$ with $1 \in B$ give rise to generators of $J$ which are variables $p_{i,j}$ with $i \in [d]$ and $1 < j \le k$.

If on the other hand $[A \mid B] \in G$ is an $\ell$-minor with $1 \in B$ then recall that $\lv B \cap C_j \rv \le 1$. This minor gives rise to a generator of $J$ of the form $[\tilde{A} \mid \tilde{B}]$ where $\tilde{A}\subset[d]$, $\tilde{B}\subset {R_i \backslash C_1}$,  $\lv \tilde{B} \cap C_j \rv \le 1$ and $|A|=|B|={\ell-1}$.

Finally if $[A \mid B] \in G$ is an $\ell$-minor with $1 \not\in B$. We have that $B \cap C_1 \neq \emptyset$ so let $c \in B \cap C_1$.
All variables $p_{i,c}$ for $i \in [d]$ belong to $J$, so $[A \mid B]$ is a redundant generator of~$J$.
This proves that the ideal $J$ is generated by the set of minors listed in the lemma.

Now let,
$Q_1 = \left\la p_{i,j} : i \in [d], 1 < j \le k \right\ra$.
and,
\begin{align*}
    Q &= \left\la [A \mid B] : A\subset [d], B\subset C_j, j \ge 2, |A|=|B|=2 \right\ra \\
      &+ \left\la [A \mid B] : A\subset[d], \tilde{B}\subset[k\ell] \backslash C_1, |A|=|\tilde{B}|={\ell-1}, \lv \tilde{B} \cap C_j \rv \le 1 \right\ra.
\end{align*}
Note that $Q_1$ is a monomial prime ideal (i.e.\ $Q_1$ is generated by variables).
Since the variables appearing in the generators of $Q_1$ and $Q$ are disjoint, we deduce that $I_0$ is prime if and only if $Q$ is prime.
The statement now follows from the observation that $Q$ is exactly $I_0$ with parameters for $(k,\ell,d,s,t)$ given by $(k,\ell-1,d,s,t-1)$.
\end{proof}

\begin{Lemma}
\label{base case}
$I_0$ is prime for $k \ge 2, \ell = 2, d \ge 2, s = 2, t = 2$.
\end{Lemma}

\begin{proof}
  In this case, $I_0$ is a (generalised) binomial edge ideal generated by all $2$-minors of~$P$.
  Primeness follows from the results in~\cite{Rauh}, but also follows easily from the following argument.
  First note that $p_{k,1}$ is non-zerodivisor modulo $I_0$ by Lemma~\ref{localise vars}.  By the Localization Lemma~\ref{rules}, $(R/I_0)_{p_{k,1}} \cong (R / J)_{p_{k,1}}$ where $J$ is a monomial ideal generated by variables. In particular $J$ is prime and so $I_0$ is prime.
\end{proof}

%


\section{Further questions}\label{sec:gen}
The first generalization we studied is for $s=3$.  We present two examples: the first one is verified by {\tt Bertini} \cite{Bertini} and the second one by {\tt Singular} \cite{Singular}.
\begin{Example}
  \label{ex:3-3}
  Let $k = \ell = s = t = d = 3$. Then $\MY$ and $P$ are as in Example~\ref{ex:s2td3}, and
  \begin{equation*}
    \Delta = \big\{
       123, 456, 789, 147, 258, 369
    \big\}
  \end{equation*}
The ideal of $\Delta^{3,3}$ is generated by nine $3$-minors.  Using {\tt Bertini} we verified that $\ID$ is primary. The Gr\"obner basis of $J_{[3],\Delta^{3,3}}$ with respect to the lex order is squarefree, which implies that $J_{[3],\Delta^{3,3}}$ is a radical ideal. Hence, $J_{[3],\Delta^{3,3}}$ is prime.
\end{Example}

\begin{Example}
Let $k=s = d = t = 3$, as in Example~\ref{ex:3-3}, and let $\ell=4$.
The ideal of $\Delta^{3,3}$ has $16$ generators of degree 3.
It is radical and its prime decomposition has been computed in {\tt Singular} using various computational techniques, see \cite{andreas}.  An earlier experiment using {\tt Bertini} failed, as the computation did not terminate. 
  
  It is shown in \cite{andreas} that the ideal of $\Delta^{3,3}$ has two prime components: 
  \begin{itemize}
  \item The first component is generated by all $3$-minors of the $3\times 12$ matrix $P$.
  \item The second component has $44$ generators of which $16$ are the original generators of the ideal of $\Delta^{3,3}$.
    The remaining $28$ generators are all homogeneous of degree $12$.
  \end{itemize}
  In particular, the second component is not a determinantal ideal in the sense that it is not generated by minors of the matrix~$P$: since $P$ has only three rows, all such minors would have degree at most three, but up to degree three, the second component agrees with the original ideal $I_{\Delta^{3,3}}$.
Moreover, neither component contains a variable, so
both components are ``important'' in the sense of~\cite{Sturmfels02:Solving_polynomial_equations}.  This implies that
our results do not generalize easily.

Further investigation of these types of components, which can be found in \cite{MatroidsCIStatements, clarke2021conditional}, shows that they are ideals of matroid varieties. In this example, the second component is obtained by saturating $J_{[3], \Delta^{3,3}}$ with respect to polynomials corresponding to the bases of the so-called `minimal matroid' of $\Delta^{3,3}$.

\end{Example}
The second direction for generalizing our results is to take $t<\ell$. Here, we mention the largest example we could compute with {\tt Singular}.

\begin{Example}
Let $k = 2 = s, \ell = 4$ and $t = d = 3$. Then we have,
$$\MY = 
\begin{pmatrix}
1   &3  &5  &7  \\
2   &4  &6  &8
\end{pmatrix}\quad \text{and}\quad P=\begin{pmatrix}
    p_{1,1}     &p_{1,2}    &p_{1,3}   &p_{1,4}    &p_{1,5}    &p_{1,6}    &p_{1,7}    &p_{1,8}    \\
    p_{2,1}     &p_{2,2}    &p_{2,3}   &p_{2,4}    &p_{2,5}    &p_{2,6}    &p_{2,7}    &p_{2,8}    \\
    p_{3,1}     &p_{3,2}    &p_{3,3}   &p_{3,4}    &p_{3,5}    &p_{3,6}    &p_{3,7}    &p_{3,8}    \\
\end{pmatrix}.
$$
The ideal $J_{[3],\Delta}$ is radical and has $4$ isomorphism classes of associated prime ideals. We call these isomorphism classes type $1$ to $4$. The number of prime ideals of each type is:
 \smallskip

\begin{center}
    \begin{tabular}{ccccc}
    \toprule
    Type    &Representative     &Occurrences    &Number of Generators     &Codimension  \\
    \midrule
    1       &$I_{146}$   &24            &14             &12         \\
    2       &$I_{1368}$  &6             &12             &12         \\
    3       &$I_{14}^*$  &12            &24             &12         \\
    4       &$I_0$       &1             &44             &14         \\
    \bottomrule
    \end{tabular}
\end{center}

The prime components are not necessarily of type $I_S$. For example, $I_{14}^*$ is the hyperedge ideal associated to
$\Delta_{14}^* = \{1, 4, 56, 57, 58, 67, 68, 78 \},$
that is,
\begin{align*}
    I_{14}^* = \la &p_{11}, p_{21}, p_{31}, p_{14}, p_{24}, p_{34}, \\
    &[12|56], [13|56], [23|56], [12|57], [13|57], [23|57], [12|58], [13|58], [23|58], \\
    &[12|67], [13|67], [23|67], [12|68], [13|68], [23|68], [12|78], [13|78], [23|78] \ra.
\end{align*}
Recall that $I_{14}$ is defined as
\begin{align*}
    I_{14} = \la &p_{11}, p_{21}, p_{31}, p_{14}, p_{24}, p_{34}, 
    [12|56], [13|56], [23|56], [12|78], [13|78], [23|78], \\
    &[235], [236], [237], [238],
    [257], [258], [267], [268], [357], [358], [367], [368] \ra.
\end{align*}
Thus $I_{14} \not\subseteq I_{14}^*$ and $I_{14} \not\supseteq I_{14}^*$. We found that $I_0\subset I_{14}$ which implies that $I_{14}$ is not a prime component of $J_{[3],\Delta}$.
\end{Example}

\section*{Acknowledgement}

\noindent
We are grateful to the anonymous reviewers for their comments. We thank Harshit Motwani for helpful comments on a preliminary version of this manuscript. 
The first author was supported by EPSRC Doctoral Training Partnership (DTP) award EP/N509619/1. The second author was partially supported by EPSRC Early Career Fellowship EP/R023379/1 and BOF Starting Grant from Ghent University.

\bibliographystyle{alpha}
\bibliography{Det.bib}

\newpage

\appendix
\section{Initial terms for Theorem~\ref{thm:Ollie}}

\subsection{Case 3.i.} \label{tab:ini-3i}

Table~\ref{tab:3i} lists the initial terms for the products of determinants that appear in Equation~\eqref{eqn:spoly1} on page~\pageref{eqn:spoly1} in Case 3.i. Note that
$$S(g_1, g_2) = \init(h)g_1 -  p_{b,a+1}g_2$$
where $h = [[n] \backslash \{a,b \} | [n] \backslash \{ a, a+1 \}]$. Since the second largest terms of $p_{b,a+1}g_2$ and $\init(h)g_1$ do not cancel, the initial term of $S(g_1, g_2)$ is the second largest term of either $p_{b,a+1}g_2$ or $\init(h)g_1$. Determining which of $p_{b,a+1}g_2$ or $\init(h)g_1$ gives rise to the initial term of $S(g_1, g_2)$ is dependent on $a$ and $b$. For each case, we make note of whether the initial term of $S(g_1, g_2)$ is taken from `(1)' or `(2)', i.e.\ $\init(h)g_1$ or $p_{b,a+1}g_2$ respectively.

The terms in the third column are the second largest terms of the given part of the equation. These are important because these second largest terms become the leading terms of their respective parts once Equation~\eqref{eqn:spoly1} is rearranged to form the reduction of the $S$-polynomial.


\subsection{Case 3.ii.} \label{tab:ini-3ii}

Similarly, Table~\ref{tab:3ii} presents the initial terms for terms appearing in Equation~\eqref{eqn:spoly1} on page~\pageref{eqn:spoly1} in Case 3.ii.  Note that for this case $S(g_1, g_2) = \init(h)g_1 - p_{b,a}g_2$ where $h = [[n] \backslash \{b, a\} | [n] \backslash \{ a, a+1 \}]$.

\subsection{Case 3.iii.} \label{tab:ini-3iii}

Table~\ref{tab:3iii} lists the initial terms for terms appearing in Equation~\eqref{eqn:spoly3} on page~\pageref{eqn:spoly3} in Case 3.iii.  For this case, we have $m = n+1$ and $S(g_1, g_2) = \init(h)g_1 - p_{b,a+1}g_2$ where $h = [[n] \backslash \{a\} | [m] \backslash \{ a, a+1 \}]$.  Similar to Case 3.i, we note for each case of $a$ and $b$, whether the initial term of $S(g_1, g_2)$ is taken from `(1)' or `(2)', i.e.\ $\init(h)g_1$ or $p_{b,a+1}g_2$ respectively.

\subsection{Case 3.iv.} \label{tab:ini-3iv}

Table~\ref{tab:3iv} lists the initial terms for terms appearing in Equation~\eqref{eqn:spoly3} on page~\pageref{eqn:spoly3} in Case 3.iv.  For this case we have $m = n+1$ and $S(g_1, g_2) = \init(h)g_1 - p_{b,a}g_2$ where $h = [[n] \backslash \{a\} | [m] \backslash \{ a, a+1 \}]$. Similar to Case 3.i, we note for each case of $a$ and $b$, whether the initial term of $S(g_1, g_2)$ is taken from `(1)' or `(2)', i.e.\ $\init(h)g_1$ or $p_{b,a}g_2$ respectively.

\begin{landscape}

\begin{table}
  \centering
  \resizebox{\linewidth}{!}{
\begin{tabular}{lll} 
    \toprule
    Part of Equation~\eqref{eqn:spoly1} &
    (Conditions): Initial term, with respect to conditions. &
    (Conditions): Second largest term, with respect to conditions when $i = a$.  \\ 
    \midrule
    $p_{b,a}[[n] \backslash \{b\} \mid [n] \backslash \{a\}]$ &
    $p_{1,1}\dots p_{a-1,a-1}p_{a,a+1}\dots p_{b-1,b} p_{b,a} p_{b+1,b+1}\dots p_{n,n}$ \\ 
    \midrule
    $p_{b,a+1}[[n] \backslash \{b\} \mid [n] \backslash \{a+1\}] = p_{b,a+1}g_{2}$ &
    $p_{1,1}\dots p_{a,a} p_{a+1,a+2}\dots p_{b-1,b}p_{b,a+1}p_{b+1,b+1}\dots p_{n,n}$ \\ 
    \midrule
    $[ib \mid a(a+1)][[n] \backslash \{i,b\} \mid [n] \backslash \{a,a+1\}]$ &
        ($i < a$) : $p_{1,1}\dots p_{i-1,i-1}p_{i,a}p_{i+1,i}\dots p_{a,a-1}p_{a+1,a+2}\dots p_{b-1,b}p_{b,a+1}p_{b+1,b+1}\dots p_{n,n}$ \\
 &        ($i = a$) : $p_{1,1}\dots p_{a,a}p_{a+1,a+2}\dots p_{b-1,b}p_{b,a+1}p_{b+1,b+1}\dots p_{n,n}$ 
 &
        ($b < n-1$) : $p_{1,1}\dots p_{a,a} p_{a+1,a+2} \dots p_{b-1, b} p_{b,a+1} p_{b+1,b+1}, \dots p_{n-2,n-2} p_{n-1,n} p_{n,n-1}$\\
 &&       ($b = n-1, a < n-2$) : $p_{1,1}\dots p_{a,a} p_{a+1,a+2} \dots p_{n-3, n-2} p_{n-2,n} p_{n-1,a+1} p_{n,n-1}$\\
 &&       ($b = n-1, a = n-2$) : $p_{1,1}\dots p_{n-3,n-3} p_{n-2,n-1} p_{n-1, n-2} p_{n,n}$\\
 &&       ($b = n, a < n-2$) : $p_{1,1}\dots p_{a,a} p_{a+1,a+2} \dots p_{n-3, n-2} p_{n-2,n} p_{n-1,n-1} p_{n,a+1}$\\
 &&       ($b = n, a = n-2$) : $p_{1,1}\dots p_{n-3, n-3} p_{n-2,n-1} p_{n-1,n} p_{n,n-2}$\\
 &&       ($b = n, a = n-1$) : $p_{1,1}\dots p_{n-2, n-2} p_{n-1,n} p_{n,n-1}$\\
 &       ($a < i < b$) : $p_{1,1}\dots p_{a-1,a-1}p_{a,a+2}\dots,p_{i-1,i+1}p_{i,a}p_{i+1,i+2}\dots p_{b-1,b}p_{b,a+1}p_{b+1,b+1}\dots p_{n,n}$ \\
    \midrule
    $[bi \mid a(a+1)][[n] \backslash \{b,i\} \mid [n] \backslash \{a,a+1\}]$ &
    ($i > b$) : $p_{1,1}\dots p_{a-1,a-1}p_{a,a+2}\dots p_{b-1,b+1}p_{b,a}p_{b+1,b+2}\dots p_{i-1,i}p_{i,a+1}p_{i+1,i+1}\dots p_{n,n}$ \\
    \midrule
    $p_{b,a+1}g_2 - \init(h)g_1 = S(g_{1},g_{2})$ &
        ($b < n-1$): $p_{1,1}\dots p_{a,a}p_{a+1,a+2}\dots p_{b-1,b}p_{b,a+1}p_{b+1,b+1}\dots p_{n-2,n-2}p_{n-1,n}p_{n,n-1}$, (2) \\ 
  &      ($b = n-1, a < n-2$): $p_{1,1}\dots p_{a,a}p_{a+1,a+2}\dots p_{n-3,n-2} p_{n-2,n}p_{n-1,a+1}p_{n,n-1}$, (2) \\ 
  &      ($b = n-1, a = n-2$): $p_{1,1}\dots p_{n-3,n-3}p_{n-2,n-1}p_{n-1,n-2}p_{n,n}$, (1) \\ 
  &      ($b = n, a < n-2$): $p_{1,1}\dots p_{a,a}p_{a+1,a+2}\dots p_{n-3,n-2} p_{n-2,n}p_{n-1,n-1}p_{n,a+1}$, (2) \\ 
  &      ($b = n, a = n-2$): $p_{1,1}\dots p_{n-3,n-3}p_{n-2,n-1}p_{n-1,n}p_{n,n-2}$, (1) \\ 
  &      ($b = n, a = n-1$): $p_{1,1}\dots p_{n-2,n-2}p_{n-1,n}p_{n,n-1}$, (1) \\ 
    \bottomrule
\end{tabular}
}
\caption{The initial terms in case 3.i.$^{\strut}$} 
\label{tab:3i}
\end{table}

\begin{table}
  \centering
  \resizebox{\linewidth}{!}{
\begin{tabular}{lll} 
    \toprule
    Part of Equation~\eqref{eqn:spoly1} &
    (Conditions): Initial term, with respect to conditions. &
    (Conditions): Second largest term, with respect to conditions when $i = a + 1$.  \\ \midrule
    $p_{b,a}[[n] \backslash \{b\} \mid [n] \backslash \{a\}] = p_{b,a}g_{2}$ &
    $p_{1,1}\dots p_{b-1,b-1}p_{b,a}p_{b+1,b}\dots p_{a,a-1}p_{a+1,a+1}\dots p_{n,n}$ \\ \midrule
    $p_{b,a+1}[[n] \backslash \{b\} \mid [n] \backslash \{a+1\}]$ &
    $p_{1,1}\dots p_{b-1,b-1}p_{b,a+1}p_{b+1,b}\dots p_{a+1,a}p_{a+2,a+2}\dots p_{n,n}$ \\ \midrule
    $[ib \mid a(a+1)][[n] \backslash \{i,b\} \mid [n] \backslash \{a,a+1\}]$ &
    ($i < b$) : $p_{1,1}\dots p_{i-1,i-1}p_{i,a}p_{i+1,i}\dots p_{b-1,b-2}p_{b,a+1}p_{b+1,b-1}\dots p_{a+1,a-1}p_{a+2,a+2}\dots p_{n,n}$ \\ \midrule
    $[bi \mid a(a+1)][[n] \backslash \{b,i\} \mid [n] \backslash \{a,a+1\}]$ &
        ($b < i \le a$) : $p_{1,1}\dots p_{b-1,b-1}p_{b,a}p_{b+1,b}\dots p_{i-1,i-2}p_{i,a+1}p_{i+1,i-1}\dots p_{a+1,a-1}p_{a+2,a+2}\dots p_{n,n}$ \\
    &    ($i = a+1$) : $p_{1,1}\dots p_{b-1,b-1}p_{b,a}p_{b+1,b}\dots p_{a,a-1}p_{a+1,a+1}p_{a+2,a+2}\dots p_{n,n}$ &
        
        ($a < n-2$) : $p_{1,1}\dots p_{b-1,b-1}p_{b,a}p_{b+1,b}\dots p_{a,a-1}p_{a+1,a+1}\dots p_{n-2,n-2}p_{n-1,n}p_{n,n-1}$\\
    &&    $(a = n-2, b < n-2)$ : $p_{1,1}\dots p_{b-1,b-1}p_{b,n-2}p_{b+1,b}\dots p_{n-3,n-4}p_{n-2,n}p_{n-1,n-1}p_{n,n-3}$\\
    &&    $(a = n-2, b = n-2)$ : $p_{1,1}\dots p_{n-3,n-3} p_{n-2,n-1} p_{n-1, n-2} p_{n,n}$\\
    &&    $(a = n-1, b < n-2)$ : $p_{1,1}\dots p_{b-1,b-1}p_{b,n-1}p_{b+1,b}\dots p_{n-3,n-4}p_{n-2,n-2}p_{n-1,n-3}p_{n,n}$\\
    &&    $(a = n-1, b = n-2)$ : $p_{1,1}\dots p_{n-3, n-3} p_{n-2,n} p_{n-1,n-2} p_{n,n-1}$\\
    &&    $(a = n-1, b = n-1)$ : $p_{1,1}\dots p_{n-2, n-2} p_{n-1,n} p_{n,n-1} $\\
    & ($i > a+1$) : $p_{1,1}\dots p_{b-1,b-1}p_{b,a}p_{b+1,b}\dots p_{a,a-1}p_{a+1,a+2}\dots p_{i-1,i}p_{i,a+1}p_{i+1,i+1}\dots p_{n,n}$ \\
    \midrule
    $p_{b,a}g_2 - \init(h)g_1 = S(g_{1},g_{2})$ &
    $p_{1,1}\dots p_{b-1,b-1}p_{b,a}p_{b+1,b}\dots p_{a,a-1}p_{a+1,a+1}\dots p_{n-2,n-2}p_{n-1,n}p_{n,n-1}$ \\
    \bottomrule
\end{tabular}
}
\caption{The initial terms in case 3.ii.$^{\strut}$} 
\label{tab:3ii}
\end{table}

\begin{table}
  \centering
  \resizebox{\linewidth}{!}{
\begin{tabular}{lll} 
    \toprule
    Part of Equation~\eqref{eqn:spoly3}& 
    (Conditions): Initial term, with respect to conditions. & 
    (Conditions): Second largest term, with respect to conditions when $i = a$.  \\ 
    \midrule
    $p_{b,a}[[n] \mid [m] \backslash \{a\}]$ &
    $p_{1,1}\dots p_{a-1,a-1}p_{a,a+1}\dots p_{b-1,b} p_{b,a} p_{b,b+1}\dots p_{n,n+1}$ \\ 
    \midrule
    $p_{b,a+1}[[n] \mid [m] \backslash \{a+1\}] = p_{b,a+1}g_{2}$ &
    $p_{1,1}\dots p_{a,a} p_{a+1,a+2}\dots p_{b-1,b}p_{b,a+1}p_{b,b+1}\dots p_{n,n+1}$ \\ 
    \midrule
    $[ib \mid a(a+1)][[n] \backslash \{i\} \mid [m] \backslash \{a,a+1\}]$ &
        $(i < a)$ : $p_{1,1}\dots p_{i-1,i-1}p_{i,a}p_{i+1,i}\dots p_{a,a-1}p_{a+1,a+2}\dots p_{b-1,b}p_{b,a+1}p_{b,b+1}\dots p_{n,n+1}$ \\
    &    $(i = a)$ : $p_{1,1}\dots p_{a,a}p_{a+1,a+2}\dots p_{b-1,b}p_{b,a+1}p_{b,b+1}\dots p_{n,n+1}$ \\
        
    &&    $(b < n-1)$ : $p_{1,1}\dots p_{a,a} p_{a+1,a+2} \dots p_{b-1, b} p_{b,a+1} p_{b,b+1}, \dots p_{n-2,n-1} p_{n-1,n+1} p_{n,n}$\\
    &&    $(b = n-1, a < n-2)$ : $p_{1,1}\dots p_{a,a} p_{a+1,a+2} \dots p_{n-2, n-1} p_{n-1,a+1} p_{n-1,n+1} p_{n,n}$\\
    &&    $(b = n-1, a = n-2)$ : $p_{1,1}\dots p_{n-2,n-2} p_{n-1,n-1} p_{n-1, n+1} p_{n,n}$\\
    &&    $(b = n, a < n-2)$ : $p_{1,1}\dots p_{a,a} p_{a+1,a+2} \dots p_{n-2, n-1} p_{n-1,n+1} p_{n,a+1} p_{n,n}$\\
    &&    $(b = n, a = n-2)$ : $p_{1,1}\dots p_{n-2, n-2} p_{n-1,n+1} p_{n,n-1} p_{n,n}$\\
    &&    $(b = n, a = n-1)$ : $p_{1,1}\dots p_{n-2, n-2} p_{n-1,n} p_{n,n-1} p_{n,n+1}$\\
    &   $(a < i < b)$ : $p_{1,1}\dots p_{a-1,a-1}p_{a,a+2}\dots,p_{i-1,i+1}p_{i,a}p_{i+1,i+2}\dots p_{b-1,b}p_{b,a+1}p_{b,b+1}\dots p_{n,n+1}$ \\
    \midrule
    $[bi \mid a(a+1)][[n] \backslash \{i\} \mid [m] \backslash \{a,a+1\}]$ &
    $(i > b)$ : $p_{1,1}\dots p_{a-1,a-1}p_{a,a+2}\dots p_{b-1,b+1}p_{b,a}p_{b,b+2}\dots p_{i-1,i+1}p_{i,a+1}p_{i+1,i+2}\dots p_{n,n+1}$ \\ 
    \midrule
    $p_{b,a+1}g_2 - \init(h)g_1 = S(g_{1},g_{2}):$ &
        $(b < n-1)$ : $p_{1,1}\dots p_{a,a}p_{a+1,a+2}\dots p_{b-1,b}p_{b,a+1}p_{b,b+1}\dots p_{n-2,n-1}p_{n-1,n+1}p_{n,n}$, (2)\\ 
    &    $(b = n-1, a < n-2)$ : $p_{1,1}\dots p_{a,a}p_{a+1,a+2}\dots p_{n-2,n-1} p_{n-1,a+1}p_{n-1,n+1} p_{n,n}$, (2) \\ 
    &    $(b = n-1, a = n-2)$ : $p_{1,1}\dots p_{n-2,n-2}p_{n-1,n-1}p_{n-1,n+1}p_{n,n}$, (2) \\ 
    &    $(b = n, a < n-1)$ : $p_{1,1}\dots p_{a,a}p_{a+1,a+2}\dots p_{n-2,n-1} p_{n-1,n+1}p_{n,a+1}p_{n,n}$, (2) \\ 
    &    $(b = n, a = n-1)$ : $p_{1,1}\dots p_{n-2,n-2}p_{n-1,n}p_{n,n-1}p_{n,n+1}$, (1) \\ 
    \bottomrule
\end{tabular}
}\caption{The initial terms in case 3.iii.$^{\strut}$} 
\label{tab:3iii}
\end{table}

\begin{table}
  \centering
  \resizebox{\linewidth}{!}{
\begin{tabular}{lll}
    \toprule
    Part of Equation~\eqref{eqn:spoly3} &
    (Conditions): Initial term, with respect to conditions. &
    (Conditions): Second largest term, with respect to conditions when $i = a$.  \\ 
    \midrule
    $p_{b,a}[[n] \mid [m] \backslash \{a\}] = p_{b,a}g_{2}$ &
    $p_{1,1}\dots p_{b,b}p_{b,a}p_{b+1,b+1}\dots p_{a-1,a-1} p_{a,a+1}\dots p_{n,n+1}$ \\ 
    \midrule
    $p_{b,a+1}[[n] \mid [m] \backslash \{a+1\}]$ &
    $p_{1,1}\dots p_{b,b}p_{b,a+1}p_{b+1,b+1}\dots p_{a,a} p_{a+1,a+2}\dots p_{n,n+1}$ \\ 
    \midrule
    $[ib \mid a(a+1)][[n] \backslash \{i\} \mid [m] \backslash \{a,a+1\}]$ &
    $(i < b)$ : $p_{1,1}\dots p_{i-1,i-1}p_{i,a}p_{i+1, i}\dots p_{b,b-1}p_{b,a+1}p_{b+1,b}\dots p_{a,a-1}p_{a+1,a+2}\dots p_{n,n+1}$ \\ 
    \midrule
    $[bi \mid a(a+1)][[n] \backslash \{i\} \mid [m] \backslash \{a,a+1\}]$ &
        $(b < i < a)$ : $p_{1,1}\dots p_{b,b}p_{b,a}p_{b+1,b+1}\dots p_{i-1,i-1}p_{i,a+1}p_{i+1,i}\dots p_{a,a-1}p_{a+1,a+2}\dots p_{n,n+1}$ \\
    &    $(i = a)$ : $p_{1,1}\dots p_{b,b}p_{b,a}p_{b+1,b+1}\dots p_{a-1,a-1}p_{a,a+1}\dots p_{n,n+1}$
    &    $(a < n-1)$ : $p_{1,1}\dots p_{b,b}p_{b,a}p_{b+1,b+1}\dots p_{a-1,a-1}p_{a,a+1}\dots p_{n-2,n-1}p_{n-1,n+1}p_{n,n}$\\
    &&    $(a = n-1, b < n-2)$ : $p_{1,1}\dots p_{b,b}p_{b,n-1}p_{b+1,b+1}\dots p_{n-3,n-3}p_{n-2,n+1}p_{n-1,n}p_{n,n-2}$\\
    &&    $(a = n-1, b = n-2)$ : $p_{1,1}\dots p_{n-2,n-2} p_{n-2,n} p_{n-1, n-1} p_{n,n+1}$\\
    &&    $(a = n, b < n-2)$ : $p_{1,1}\dots p_{b,b} p_{b,n} p_{b+1,b+1} \dots p_{n-3, n-3} p_{n-2,n-1} p_{n-1,n-2} p_{n,n+1}$\\
    &&    $(a = n, b = n-2)$ : $p_{1,1}\dots p_{n-2, n-2} p_{n-2,n+1} p_{n-1,n-1} p_{n,n}$\\
    &&    $(a = n, b = n-1)$ : $p_{1,1}\dots p_{n-1, n-1} p_{n-1,n+1} p_{n,n}$\\
    &   $(i > a)$ : $p_{1,1}\dots p_{b,b}p_{b,a}p_{b+1,b+1}\dots,p_{a-1,a-1}p_{a,a+2}\dots p_{i-1,i+1}p_{i,a+1}p_{i+1,i+2}\dots p_{n,n+1}$ \\
    \midrule
    $p_{b,a}g_2 - \init(h)g_1 = S(g_{1},g_{2})$ &
        $(a < n-1)$ : $p_{1,1}\dots p_{b,b}p_{b,a}p_{b+1,b+1}\dots p_{a-1,a-1}p_{a,a+1}\dots p_{n-2,n-1}p_{n-1,n+1}p_{n,n}$, (2) \\ 
    &    $(a = n-1, b < n-2)$ : $p_{1,1}\dots p_{b,b}p_{b,\textcolor{red}{n-1}}p_{b+1,b+1}\dots p_{n-2,n-2}p_{n-1,n+1}p_{n,n}$, (2) \\ 
    &    $(a = n-1, b = n-2)$ : $p_{1,1}\dots p_{n-2,n-2}p_{n-2,n-1}p_{n-1,n+1}p_{n,n}$, (2) \\ 
    &    $(a = n, b < n-1)$ : $p_{1,1}\dots p_{b,b}p_{b,n}p_{b+1,b+1}\dots p_{n-2,n-2} p_{n-1,n+1}p_{n,n-1}$, (2) \\ 
    &    $(a = n, b = n-1)$ : $p_{1,1}\dots p_{n-1,n-1}p_{n-1,n+1}p_{n,n}$, (1)\\ 
    \bottomrule
\end{tabular}
}
\caption{The initial terms in case 3.iv.$^{\strut}$} 
\label{tab:3iv}
\end{table}

\end{landscape}

\end{document}